\documentclass[11pt,a4paper]{amsart}[2000]
\usepackage{geometry}
\usepackage{amsmath}
\usepackage{amssymb}
\usepackage{amsthm}
\usepackage[hang]{footmisc}
\usepackage[all]{xypic}
\usepackage{hyperref}
\usepackage{color}
\usepackage{mathtools}
\usepackage{enumitem}
\usepackage{stackengine}

\title[Fiberwise amenability of ample \'{e}tale groupoids]{Fiberwise amenability of ample \'{e}tale groupoids}

\thanks{}

\allowdisplaybreaks

\theoremstyle{plain}

\newtheorem{Thm}{Theorem}[section]

\theoremstyle{definition}

\newtheorem{Exl}[Thm]{Example}

\theoremstyle{plain}
\newtheorem{thm}[Thm]{Theorem}
\newtheorem{lem}[Thm]{Lemma}
\newtheorem{cor}[Thm]{Corollary}
\newtheorem{prop}[Thm]{Proposition}

\theoremstyle{definition}
\newtheorem{defn}[Thm]{Definition}

\newtheorem{eg}[Thm]{Example}
\newtheorem{rmk}[Thm]{Remark}

\newtheorem{innercustomthm}{Theorem}
\newenvironment{customthm}[1]
{\renewcommand\theinnercustomthm{#1}\innercustomthm}
{\endinnercustomthm}

\newtheorem{innercustomcor}{Corollary}
\newenvironment{customcor}[1]
{\renewcommand\theinnercustomcor{#1}\innercustomcor}
{\endinnercustomcor}

\newcommand\barbelow[1]{\stackunder[1.2pt]{$#1$}{\rule{.8ex}{.075ex}}}

\newcommand{\B}{B}
\newcommand{\A}{A}
\newcommand{\J}{J}
\newcommand{\K}{\mathcal{K}}
\newcommand{\D}{D}
\newcommand{\Ch}{D}
\newcommand{\Zh}{\mathcal{Z}}
\newcommand{\E}{E}
\newcommand{\Oh}{\mathcal{O}}

\newcommand{\T}{{\mathbb T}}
\newcommand{\R}{{\mathbb R}}
\newcommand{\N}{{\mathbb N}}
\newcommand{\Z}{{\mathbb Z}}
\newcommand{\C}{{\mathbb C}}
\newcommand{\Q}{{\mathbb Q}}

\newcommand{\aut}{\mathrm{Aut}}
\newcommand{\supp}{\mathrm{supp}}

\newcommand{\eps}{\varepsilon}
\numberwithin{equation}{section}
\newcommand{\map}[0]{\operatorname{Map}}

\newcommand{\id}{\mathrm{id}}

\newcommand{\halpha}{\widehat{\alpha}}
\newcommand{\calpha}{\widehat{\alpha}}

\newcommand{\tih}{\widetilde {h}}

\newcommand\set[1]{\left\{#1\right\}}  
\newcommand\mset[1]{\left\{\!\!\left\{#1\right\}\!\!\right\}}

\newcommand{\CA}[0]{\mathcal{A}} \newcommand{\CB}[0]{\mathcal{B}}
\newcommand{\CC}[0]{\mathcal{C}} \newcommand{\CD}[0]{\mathcal{D}}
 \newcommand{\CF}[0]{\mathcal{F}}
\newcommand{\CG}[0]{\mathcal{G}} \newcommand{\CH}[0]{\mathcal{H}}

\newcommand{\CO}[0]{\mathcal{O}} \newcommand{\CP}[0]{\mathcal{P}}
\newcommand{\CQ}[0]{\mathcal{Q}} 
\newcommand{\CS}[0]{\mathcal{S}} \newcommand{\CT}[0]{\mathcal{T}}
 \newcommand{\CV}[0]{\mathcal{V}}
 
 \newcommand{\CZ}[0]{\mathcal{Z}}

\newcommand{\Ra}[0]{\Rightarrow}
\newcommand{\La}[0]{\Leftarrow}
\newcommand{\LRa}[0]{\Leftrightarrow}

\newcommand{\quer}[0]{\overline}
\newcommand{\eins}[0]{\mathbf{1}}			
\newcommand{\diag}[0]{\operatorname{diag}}
\newcommand{\ad}[0]{\operatorname{Ad}}
\newcommand{\ev}[0]{\operatorname{ev}}
\newcommand{\fin}[0]{{\subset\!\!\!\subset}}
\newcommand{\diam}[0]{\operatorname{diam}}
\newcommand{\Hom}[0]{\operatorname{Hom}}
\newcommand{\dst}[0]{\displaystyle}
\newcommand{\spp}[0]{\operatorname{supp}}
\newcommand{\lsc}[0]{\operatorname{Lsc}}
\newcommand{\del}[0]{\partial}
\newcommand{\GU}[0]{\CG^{(0)}}

\theoremstyle{definition}

\numberwithin{equation}{Thm}

\begin{document}
\global\long\def\floorstar#1{\lfloor#1\rfloor}
\global\long\def\ceilstar#1{\lceil#1\rceil}	

\global\long\def\B{B}
\global\long\def\A{A}
\global\long\def\J{J}
\global\long\def\K{\mathcal{K}}
\global\long\def\D{D}
\global\long\def\Ch{D}
\global\long\def\Zh{\mathcal{Z}}
\global\long\def\E{E}
\global\long\def\Oh{\mathcal{O}}

\global\long\def\T{{\mathbb{T}}}
\global\long\def\BR{{\mathbb{R}}}
\global\long\def\N{{\mathbb{N}}}
\global\long\def\Z{{\mathbb{Z}}}
\global\long\def\C{{\mathbb{C}}}
\global\long\def\Q{{\mathbb{Q}}}

\global\long\def\aut{\mathrm{Aut}}
\global\long\def\supp{\mathrm{supp}}

\global\long\def\eps{\varepsilon}

\global\long\def\id{\mathrm{id}}

\global\long\def\halpha{\widehat{\alpha}}
\global\long\def\calpha{\widehat{\alpha}}

\global\long\def\tih{\widetilde{h}}

\global\long\def\opFol{\operatorname{F{\o}l}}

\global\long\def\opRange{\operatorname{Range}}

\global\long\def\opIso{\operatorname{Iso}}

\global\long\def\dimnuc{\dim_{\operatorname{nuc}}}

\global\long\def\set#1{\left\{  #1\right\}  }


\global\long\def\mset#1{\left\{  \!\!\left\{  #1\right\}  \!\!\right\}  }

\global\long\def\Ra{\Rightarrow}
\global\long\def\La{\Leftarrow}
\global\long\def\LRa{\Leftrightarrow}

\global\long\def\quer{\overline{}}
\global\long\def\eins{\mathbf{1}}
\global\long\def\diag{\operatorname{diag}}
\global\long\def\ad{\operatorname{Ad}}
\global\long\def\ev{\operatorname{ev}}
\global\long\def\fin{{\subset\!\!\!\subset}}
\global\long\def\diam{\operatorname{diam}}
\global\long\def\Hom{\operatorname{Hom}}
\global\long\def\dst{{\displaystyle }}
\global\long\def\spp{\operatorname{supp}}
\global\long\def\spo{\operatorname{supp}_{o}}
\global\long\def\del{\partial}
\global\long\def\lsc{\operatorname{Lsc}}
\global\long\def\GU{\CG^{(0)}}
\global\long\def\HU{\CH^{(0)}}
\global\long\def\AU{\CA^{(0)}}
\global\long\def\BU{\CB^{(0)}}
\global\long\def\CUU{\CC^{(0)}}
\global\long\def\DU{\CD^{(0)}}
\global\long\def\QU{\CQ^{(0)}}
\global\long\def\TU{\CT^{(0)}}
\global\long\def\CUUU{\CC'{}^{(0)}}

\global\long\def\AUl{(\CA^{l})^{(0)}}
\global\long\def\BUl{(B^{l})^{(0)}}
\global\long\def\HUp{(\CH^{p})^{(0)}}
\global\long\def\sym{\operatorname{Sym}}

\global\long\def\properlength{proper}

\global\long\def\interior#1{#1^{\operatorname{o}}}
	
\author{Xin Ma}
\email{xma1@memphis.edu}
\address{Department of Mathematics,
	University of Memphis,
	Memphis, TN, 38152}

\subjclass[2010]{22A22, 46L35, 37B05, 20E25}
\keywords{}

\date{Oct 19, 2021}

	
\begin{abstract}
Let $\CG$ be a locally compact $\sigma$-compact Hausdorff ample groupoid on a compact space. In this paper, we further examine the (ubiquitous) fiberwise amenability introduced by the author and Jianchao Wu for $\CG$. We define the corresponding concepts of F{\o}lner sequences and Banach densities for $\CG$, based on which, we establish a topological groupoid version of the Ornstein-Weiss quasi-tilling theorem.  This leads to the notion of almost finiteness in measure for ample groupoids as a weaker version of Matui's almost finiteness. As applications, we first show that $C^*_r(\CG)$ has the uniform property $\Gamma$ and thus satisfies the Toms-Winter conjecture when $\CG$ is minimal second countable (topologically) amenable and almost finite in measure.  Then we prove that the topological full group $[[\CG]]$ is always sofic when $\CG$ is second countable minimal and admits a F{\o}lner sequence. This can be used to strengthen one of Matui's result on the commutator subgroup $D[[\CG]]$ when $\CG$ is almost finite. Concrete examples are provided.
\end{abstract}
\maketitle

\section{Introduction}
In \cite{M-W}, the author and Jianchao Wu introduced a new framework including \emph{(ubiquitous) fiberwise amenability} and \emph{almost elementariness} for locally compact Hausdorff \'{e}tale groupoids where the (ubiquitous) fiberwise amenability  is a coarse geometric
property of étale groupoids that is closely related to the existence
of invariant measures on unit spaces and corresponds to the amenability
of the acting group in a transformation groupoid. Almost elementariness, on the other hand, is a finite-dimensional approximation property for groupoids.

In this paper, we further investigate these notions in the setting of locally compact Hausdorff ample \'{e}tale groupoids. We have three main motivation.  First, fiberwise amenability, as a new geometric or combinatoric property for groupoids, is of considerable independent interest. Many natural questions may be asked. For example, for fiberwise amenable groupoids as analogues of amenable groups and their actions,  what is a proper form of F{\o}lner sequences in this context? In addition, the Ornstein-Weiss (quasi-)tilling theorem for amenable groups and their free actions are very useful in study  groups and their actions as well as applicable to other areas such as the structure theory of $C^*$-algebras of crossed products.  See e.g., \cite{C-J-K-M-S-T}, \cite{D} and \cite{D-G}. The  essential way to use Ornstein-Weiss (quasi-)tilling theorem is that the (quasi-)tilling of the underlying space with arbitrary F{\o}lner shapes would yield a good c.p.c order-zero map $\varphi$ from a matrix algebra $M_n(\C)$ to the crossed product $A$. F{\o}lner shapes (and a good quasi-tiling of the shape by other F{\o}lner sets) are used to verify almost commutativity of certain elements in the algebra $A$. The small reminder condition of the quasi-tilling of the underlying space will establish that the part outside the approximation of the quasi-tilling, i.e. $1_A-\varphi(1_n)$, is small in certain senses. Using different interpretation of this ``smallness'', this leads to (tracial) $\CZ$-stability  or uniform property $\Gamma$ results, respectively.

 Motivated by these and the tracial $\CZ$-stability result for groupoid $C^*$-algebras in \cite[Theorem 9.7]{M-W} obtained using fiberwise amenability and almost elementariness,  another natural question is  when and how can we  do Ornstein-Weiss type (quasi)-tilling for (ubiquitous) fiberwise amenable groupoids?  A satisfactory answer to this question will lead to a new finite-dimensional approximation property for groupoids, which is also related to known ones such as almost finiteness and almost elementariness. Therefore this would have potential applications in detecting the structure of groupoid $C^*$-algebras, which enlarges the view because there exists $C^*$-algebras, due to $K$-theoretical obstructions, like unital AF-algebras and the Jiang-Su algebra $\CZ$,  that fail to be written in the form of crossed product but can be written as groupoid $C^*$-algebras.

In this paper, we plan to discuss these questions. Developing such a theory surrounding fiberwise amenability on general locally compact Hausdorff \'{e}tale groupoids is much more challenging than on groups and group actions.  The main difficulty comes from the fiberwise nature of groupoids, which implies that the dynamics of the groupoid on the unit space is not global homogeneous in the sense that all source fibers are geometrically same, even \'{e}taleness could provide the local homogeneity. Unlike the classical cases, the lack of the global homogeneity makes the dynamics of the groupoid on its unit space somehow ``blind'' so that one may define ``strange'' F{\o}lner sequences with too many overlaps and thus become useless, in which case,  the Ornstein-Weiss argument totally fails. This happens even for transformation groupoids (see Example \ref{eg: ill-behaved density}).  We thus try to characterize when groupoid admits such a good (quasi)-tilling.

We refer to \emph{almost finiteness in measure} for the desired property of existence of  Ornstein-Weiss quasi-tillings with arbitrary F{\o}lner fiberwise shapes in the groupoid (see Definition \ref{defn: af in measure}).  This name is directly borrowed from \cite{D-G} because in \cite{D-G} it describes the same phenomenon for free topological dynamical systems of amenable groups. On the other hand, our almost finiteness in measure is a weaker version of Matui's almost finiteness (recorded as Definition \ref{defn: almost finiteness}), which is originally designed for ample groupoids in \cite{Matui}. In this context, We obtain an almost characterization on when the groupoid is almost finite in measure as follows. First if a principal groupoid is almost finite in measure then it always admits ``good'' F{\o}lner sequences. See Definition \ref{defn: folner sequence for groupoids} for groupoid F{\o}lner sequences.
\begin{description}
	\item[Fact] (Proposition \ref{prop:almost finite in measure admits good Folner sequence}) Let $\CG$ be a locally compact Hausdorff \'{e}tale ample principal groupoid  with a compact unit space, which is almost finite in measure and $\{\epsilon_n: n\in\N\}$ a decreasing sequence such that $\lim_{n\to \infty}\epsilon_n=0$. Then $\CG$ admits a F{\o}lner sequence $\{S_n: n\in\N\}$ in which each $S_n$ is $(1, \epsilon_n)$-good in the sense of Definition \ref{defn: good folner set}.  
\end{description}
The following is a partial converse of the fact as our first main result, which asks for an additional conditions for F{\o}lner sequences.

 \begin{customthm}{A}(Theorem \ref{thm: almost finite in measre})
 	Let $\CG$ be a locally compact Hausdorff \'{e}tale ample groupoid with a compact unit space. Let $M\in \N$ and $\{\epsilon_n: n\in\N\}$ be a decreasing sequence such that $\lim_{n\to \infty}\epsilon_n=0$. Suppose $\CG$ admits a F{\o}lner sequence $\CS=\{S_n: n\in\N\}$ in which each $S_n$ is $(M, \epsilon_n)$-good and has $1$-controlled height in the sense of Definition \ref{defn: control length}.  Then $\CG$ is almost finite in measure.
 \end{customthm}
So far it is still unknown to the author whether one can drop the  ``$1$-controlled height''  condition for each $S_n$ in $\CS$ to obtain a full characterization. It is easy to see that all transformation groupoid of actions of amenable groups on the Cantor set naturally admits a F{\o}lner sequence satisfying the assumption of Theorem A. In addition, we can see in Remark \ref{rmk: 1-controlled height} that every minimal ample groupoid with a unit $u\in \GU$ such that $\CG_u^u=\{u\}$ admits a F{\o}lner sequence $\CS$ in which each $S_n$ has $1$-controlled height but in general these $S_n$ are not $(M, \epsilon_n)$-good. Nevertheless, see Example \ref{eg: partial dynamical system} that groupoids of certain partial dynamical systems admits this kind of nice F{\o}lner sequences and thus are almost finite in measure.

The second motivation is on applications to $C^*$-algebras, motivated by results in \cite{D-G} and \cite{M-W}, one can relate the almost finiteness in measure to the structure theory of $C^*_r(\CG)$ by establishing the following result, i.e., Theorem B. We remark that, similar to the situation in \cite{M-W}, one still cannot apply the ideas in \cite{D-G} to prove this theorem. The main obstruction is that, in a general fiberwise amenable groupoid $\CG$, so far there has been no topological tiling results for any  compact ``F{\o}lner'' set that is applicable to general groupoids. However,  we remark that it was developed a version in \cite{I-W-Z} that is applicable to certain tilling groupoids. Nevertheless, the methods developed in \cite{M-W} called \emph{extendability} and almost elementariness actually are good enough here to construct an c.p.c order-zero map from matrix algebras to the groupoid algebras. Instead of F{\o}lner conditions, the extendability would  create ``extra room'' so that the almost commutativity of certain elements is obtained. Then combining \cite[Theorem 5.6]{C-E-T-W}, one has the following result.

\begin{customthm}{B}(Theorem \ref{thm: uniform property Gamma} and Corollary \ref{cor: tw conjecture})
	Let $\CG$ be a second countable minimal principal ample groupoid  with a compact unit space, which is also almost finite in measure. Then its reduced $C^*$-algebra $C^*_r(\CG)$ has the uniform property $\Gamma$. If $\CG$ is additionally assumed to be (topologically) amenable, then $C^*_r(\CG)$ satisfies the Toms-Winter conjecture.
\end{customthm}

We remark that, unlike group action case in \cite{D-G}, almost finiteness in measure does not necessarily implies that the groupoid $\CG$ is (topologically) amenable. See counterexamples in \cite{A-B-B-L} and \cite{E}. Our final motivation is to apply the fiberwise amenability in studying topological full groups. Since all discrete amenable groups are sofic, it is reasonable to conjecture that the topological full groups of fiberwise amenable groupoids are sofic. Indeed, using groupoid F{\o}lner sequences, we prove the following result.

\begin{customthm}{C}(Theorem \ref{thm: sofic})
	Let $\CG$ be a locally compact Hausdorff minimal second countable ample groupoid  with a compact unit space. Suppose $\CG$ admits a F{\o}lner sequence. Then $[[\CG]]$ is sofic.
	\end{customthm}

We will see in Proposition \ref{prop: minimal groupoid normal folner} and \ref{prop: exist folner sequence} that if $\CG$ is   fiberwise amenable, minimal and satisfies a very weak freeness condition that there exists a $u\in \GU$ such that $\CG^u_u=\{u\}$, then it naturally admits a F{\o}lner sequence. So we have the following by combining \cite{Matui2} and \cite{Ne}.

\begin{customcor}{D}(Corollary \ref{cor 1} and \ref{cor 2})
	Let $\CG$ be a locally compact Hausdorff minimal second countable ample groupoid  with a compact unit space . 
\begin{enumerate}
	\item 	Suppose $\CG$ is fiberwise amenable and there is a $u\in \GU$ such that $\CG^u_u=\{u\}$. Then $[[\CG]]$ is sofic. In particular,  if $\CG$ is topological principal then the alternating group $A(\CG)$ introduced by Nekrashevych in \cite{Ne} is simple and sofic. 
	\item If $\CG$ is  almost finite then the commutator group $D[[\CG]]$ of $[[\CG]]$ is simple and sofic.  
\end{enumerate}
\end{customcor}

We remark that Matui in \cite{Matui2} showed that  $D[[\CG]]$ is simple when $\CG$ is minimal and almost finite. Thus, our Corollary D strengthens the Matui's result.

\section{Preliminaries}
In this section we recall some basic backgrounds. We refer to \cite{Renault} and \cite{Sims} for definitions of groupoids and their $C^*$-algebras. Using the terminology in \cite{Sims}, we denote by $\GU$ the  \textit{unit space} of $\CG$. We write \emph{source} and \emph{range} maps $s,r:\CG\rightarrow\GU$, respectively and they are defined by $s(\gamma)=\gamma^{-1}\gamma$ and $r(\gamma)=\gamma\gamma^{-1}$ for $\gamma\in\CG$.
When a groupoid $\CG$ is endowed with a locally compact Hausdorff
topology under which the product and inverse maps are continuous,
the groupoid $\CG$ is called a locally compact Hausdorff groupoid.
A locally compact Hausdorff groupoid $\CG$ is called \textit{étale}
if the range map $r$ is a local homeomorphism from $\CG$ to itself,
which means for any $\gamma\in\CG$ there is an open neighborhood
$U$ of $\gamma$ such that $r(U)$ is open and $r|_{U}$ is a homeomorphism.
A set $B$ is called a \textit{bisection}
if there is an open set $U$ in $\CG$ such that $B\subset U$ and
the restriction of the source map $s|_{U}:U\rightarrow s(U)$ and
the range map $r|_{U}:U\rightarrow r(U)$ on $U$ are both homeomorphisms
onto open subsets of $\GU$. It is not hard to see a locally compact
Hausdorff groupoid is étale if and only if its topology has a basis
consisting of open bisections. We say a locally compact Hausdorff
étale groupoid $\CG$ is \textit{ample} if its topology has a basis
consisting of compact open bisections.
\begin{Exl}
	\label{exa:transformation-groupoid}Let $X$ be a locally compact
	Hausdorff space and $\Gamma$ be a discrete group. Then any action
	$\Gamma\curvearrowright X$ by homeomorphisms induces a locally compact
	Hausdorff étale groupoid 
	\[
	X\rtimes\Gamma\coloneqq\{(\gamma x,\gamma,x):\gamma\in\Gamma,x\in X\}
	\]
	equipped with the relative topology as a subset of $X\times\Gamma\times X$.
	In addition, $(\gamma x,\gamma,x)$ and $(\beta y,\beta,y)$ are composable
	only if $\beta y=x$ and 
	\[
	(\gamma x,\gamma,x)(\beta y,\beta,y)=(\gamma\beta y,\gamma\beta y,y).
	\]
	One also defines $(\gamma x,\gamma,x)^{-1}=(x,\gamma^{-1},\gamma x)$
	and announces that $\GU\coloneqq\{(x,e_{\Gamma},x):x\in X\}$. It
	is not hard to verify that $s(\gamma x,\gamma,x)=x$ and $r(\gamma x,\gamma,x)=\gamma x$.
	The groupoid $X\rtimes\Gamma$ is called a \textit{transformation
		groupoid}. 
\end{Exl}

For any set $D\subset\GU$, Denote by 
\[
\CG_{D}\coloneqq\{\gamma\in\CG:s(\gamma)\in D\},\ \CG^{D}\coloneqq\{\gamma\in\CG:r(\gamma)\in D\},\ \text{and}\ \ \CG_{D}^{D}\coloneqq\CG^{D}\cap\CG_{D}.
\]
For the singleton case $D=\{u\}$, we write $\CG_{u}$, $\CG^{u}$
and $\CG_{u}^{u}$ instead for simplicity. In this situation, we call
$\CG_{u}$ a \textit{source fiber} and $\CG^{u}$ a \textit{range
	fiber}. In addition, each $\CG_{u}^{u}$ is a group, which is called
the \textit{isotropy} at $u$. We also denote by \[\opIso(\CG)=\bigcup_{u\in \GU}\CG^u_u=\{x\in \CG: s(x)=r(x)\}\] the isotropy of the groupoid $\CG$. We say a groupoid $\CG$ is \textit{principal}
if $\opIso(\CG)=\GU$. A groupoid $\CG$ is called \textit{topologically
	principal} if the set $\{u\in\GU:\CG_{u}^{u}=\{u\}\}$ is dense in
$\GU$. The groupoid $\CG$ is also said to be \textit{effective} if $\opIso(\CG)^o=\GU$. Recall that effectiveness is equivalent to topological principalness  if $\CG$ is second countable (see \cite[Lemma 4.2.3]{Sims}). A subset $D$ in $\GU$ is called $\CG$-\textit{invariant}
if $r(\CG D)=D$, which is equivalent to the condition $\CG^{D}=\CG_{D}$.
Note that $\CG|_{D}\coloneqq\CG_{D}^{D}$ is a subgroupoid of $\CG$
with the unit space $D$ if $D$ is a $\CG$-invariant set in $\GU$.
A groupoid $\CG$ is called \textit{minimal} if there are no proper
non-trivial closed $\CG$-invariant subsets in $\GU$.

The following definition of \emph{mutisections} was introduced by
Nekrashevych in \cite[Definition 3.1]{Ne}
\begin{defn}
	\label{5.6} A finite set of bisections $\CT=\{C_{i,j}:i,j\in F\}$
	with a finite index set $F$ is called a \emph{multisection} if it
	satisfies 
	\begin{enumerate}
		\item $C_{i,j}C_{j,k}=C_{i,k}$ for $i,j,k\in F$; 
		\item \{$C_{i,i}:i\in F\}$ is a disjoint family of subsets of $\CG^{(0)}$
		. 
	\end{enumerate}
	We call all $C_{i,i}$ the \text{levels} of the multisection $\CT$.
	All $C_{i,j}$ ($i\neq j$) are called \text{ladders} of the multisection
	$\CT$. 
\end{defn}
We say a multisection $\CT=\{C_{i,j}:i,j\in F\}$ \emph{open} (\emph{compact, closed})
if all bisections $C_{i,j}$ are open (compact, closed). In addition, we call
a finite disjoint family of multisections $\CC=\{\CT_{l}:l\in I\}$
a \emph{castle}, where $I$ is a finite index set. If all multisections
in $\CC$ are open (closed) then we say the castle $\CC$ is open
(closed). We also explicitly write $\CC=\{C_{i,j}^{l}:i,j\in F_{l},l\in I\}$,
which satisfies the following 
\begin{enumerate}[label=(\roman*)]
	\item $\{C_{i,j}^{l}:i,j\in F_{l}\}$ is a multisection; 
	\item $C_{i,j}^{l}C_{i',j'}^{l'}=\emptyset$ if $l\neq l'$. 
\end{enumerate}
Let $\CC=\{C_{i,j}^{l}:i,j\in F_{l},l\in I\}$ be a castle. Any certain
level in a multisection in $\CC$ is usually referred to as a $\CC$-level.
Analogously, any ladder in in a multisection in $\CC$ is usually
referred as a $\CC$-ladder. 
\begin{rmk}\label{rmk: simplify notation}
We remark that the disjoint union 
$\CH_{\CC}=\bigcup\CC=\bigsqcup_{l\in I}\bigsqcup_{i,j\in F_{l}}C_{i,j}^{l}$
of bisections in $\CC$ is an elementary groupoid. From this viewpoint, as in \cite{M-W}, we denote by $\CC^{(0)}$ the family $\{C_{i,i}^{l}:i\in F_{l},l\in I\}$ of all $\CC$-levels. To simplify the notation, we also write $\CC u$ for the fiber $\CH_\CC u$ for any $u\in \bigcup\CUU$.
\end{rmk}

\begin{defn}\label{defn: transversal}
 For multisections inside $\CC$, we usually denoted by $\CC_l=\{C_{i,j}^{l}:i,j\in F_{l}\}$ for each index $l\in I$. For each $l\in I$, choose one $i_l\in I$. 
We say the set $T=\bigsqcup_{l\in I}C^l_{i_l, i_l}$ a \textit{transversal} of $\CC$, which satisfies that $\bigsqcup_{l\in I}\{r(\CC u): u\in T\}=\CC^{(0)}$.
\end{defn}
Each $\CC$-ladder $C_{i,j}^{l}$ in $\CC_l$ for $i\neq j$
is also called a $\CC_l$-ladder and any $\CC$-level $C_{i,i}^{l}$
is also referred as a $\CC_l$-level.  Let $\CC$ and $\CD$ be two castles, we say $\CC$ is \textit{sub-castle} of $\CD$ if $\CC\subset \CD$. The following proposition is a natural analogue of  free  actions on zero-dimensional compact Hausdorff spaces.

\begin{prop}\label{build-tower}
Let $\CG$ be a groupoid. Then for any compact set $K$ in $\CG$ such that $r|_{Kv}$ is injective for any $v\in s(K)$ and a $u\in s(K)$, there is a  open multisection $\CC= \{C_{i,j}: i, j\in F\}$ and an $i_u\in F$ such that $u\in C_{i_u, i_u}$ and $\{u\}\cup Ku=\bigsqcup_{i\in F}C_{i, i_u}u$. If $\CG$ is additionally assumed to be ample, one can make $\CC$ compact open.
\end{prop}
\begin{proof}
Enumerate $Ku$ by $\{x_1, \dots, x_n\}$ and choose an open bisection $U_i$ such that $x_i\in U_i$. Now define $C_{i_u, i_u}=\bigcap_{i=1}^n s(U_i)$ and $F=\{1,\dots, n\}\cup \{i_u\}$. In addition, we announce  that $C_{i, i_u}=U_iC_{i_u, i_u}$  $C^{i_u, i}=C_{i_u, i}^{-1}$ for each $1\leq i\leq n$ and $C_{i, j}=C_{i, i_u}C_{i_u, j}$ for each $i, j\in F$. Then by our definition, one has $\CC=\{C_{i, j}: i, j\in F\}$ is an open multisection and $\{u\}\cup Ku=\{x_1,\dots, x_n\}=\bigsqcup_{i\in F}C_{i, i_u}u$. If $\CG$ is ample, one can make each $U_i$ to be compact open in the first place and then $\CC$ is a compact open multisection.
\end{proof}

It is well known that there is a $C^{*}$-algebraic embedding $\iota:C_{0}(\GU)\rightarrow C_{r}^{*}(\CG)$.
On the other hand, $E_{0}:C_{c}(\CG)\rightarrow C_{c}(\GU)$ defined
by $E_{0}(a)=a|_{\GU}$ extends to a faithful canonical conditional
expectation $E:C_{r}^{*}(\CG)\rightarrow C_{0}(\GU)$ satisfying $E(\iota(f))=f$
for any $f\in C_{0}(\GU)$ and $E(\iota(f)a\iota(g))=fE(a)g$ for
any $a\in C_{r}^{*}(\CG)$ and $f,g\in C_{0}(\GU)$.

Let $\CG$ be a locally compact Hausdorff étale groupoid. Suppose
$U$ is an open bisection and $f\in C_{c}(\CG)_{+}$ such that $\supp(f)\subset U$.
Define functions $s(f),r(f)\in C_{0}(\GU)$ by $s(f)(s(\gamma))=f(\gamma)$
and $r(f)(r(\gamma))=f(\gamma)$ for $\gamma\in\supp(f)$. Since $U$
is a bisection, so is $\supp(f)$. Then the functions $s(f)$ and
$r(f)$ are well-defined functions supported on $s(\supp(f))$ and
$r(\supp(f))$, respectively. Note that $s(f)=(f^{*}*f)^{1/2}$ and
$r(f)=(f*f^{*})^{1/2}$.

Finally, throughout the paper, we write $B\sqcup C$ to indicate that
the union of sets $B$ and $C$ is a disjoint union. In addition,
we denote by $\bigsqcup_{i\in I}B_{i}$ for the disjoint union of
the family $\{B_{i}:i\in I\}$. 

\section{Coarse geometry on groupoids and fiberwise amenability}
In this section, we recall some backgrounds on coarse geometry of $\sigma$-compact groupoids introduced by the author and Jianchao Wu in \cite{M-W}. To do this, one starts with equipping each $\sigma$-compact groupoid with a length function, which induces a metric on the groupoid. Then the groupoid can be viewed  as an extended metric space so that it is possible to discuss geometric properties of the groupoid such as amenability in the sense of Block and Weinberger in \cite{B-W}. This approach, in the same spirit of geometric group theory, seems to lead to a geometric groupoid theory.  For the convenience of readers, we list several definitions and results established in \cite{M-W} for use in the following sections. 
Let $(X,d)$ be
a metric space equipped with the metric $d$. We denote by $B_{d}(x,R)$
the open ball $B_{d}(x,R)=\{y\in X:d(x,y)<R\}$ and by $\bar{B}_{d}(x,R)$
the closed ball $\bar{B}_{d}(x,R)=\{y\in X:d(x,y)\leq R\}$. Let $A$
be a subset of $X$. We write $B_{d}(A,R)$ and $\bar{B}_{d}(A,R)$
for analogous meaning. If the metric is clear, we write $B(A,R)$
and $\bar{B}(A,R)$ instead for simplification. All groupoids in this section are locally compact Hausdorff \'{e}tale and have a compact unit space.
\subsection{Metrics on $\sigma$-compact groupoids}\label{sub3.1}
\begin{defn}(\cite[Definition 4.1]{M-W})
	\label{def:invariant-fiberwise-extended-metric}An extended metric
	on a groupoid $\CG$ is \end{defn}
\begin{itemize}
	\item \emph{invariant} (or, more precisely, \emph{right-invariant}) if,
	for any $x,y,z\in\CG$ with $s(x)=s(y)=r(z)$, we have $\rho(x,y)=\rho(xz,yz)$; 
	\item \emph{fiberwise} (or, more precisely, \emph{source-fiberwise}) if,
	for any $x,y\in\CG$, we have $\rho(x,y)=\infty$ if and only if $s(x)\neq s(y)$.
\end{itemize}
Like in the case of groups,  invariant metrics can be encoded by length functions. 
\begin{defn}(\cite[Definition 4.2]{M-W})
	\label{def:length-function}Recall a length function on a groupoid
	$\CG$ is a function $\ell:\CG\rightarrow[0,\infty)$ satisfying,
	for any $x,y\in\CG$, 
	\begin{enumerate}[label=(\roman*)]
		\item  $\ell(x)=0$ if and only if $x\in\CG^{(0)}$, 
		\item (symmetricity) $\ell(x)=\ell(x^{-1})$, and 
		\item (subadditivity) $\ell(xy)\leq\ell(x)+\ell(y)$ if $x$ and $y$ are
		composable in $\CG$. 
	\end{enumerate}
\end{defn}

It is not hard to see that there is a one-to-one correspondence 
between length functions and invariant fiberwise extended metrics by the following construction.
Given any length function $\ell$ on $\CG$, we associate
an extended metric $\rho_{\ell}$ by declaring, for $x,y\in\CG$,
\[
\rho_{\ell}(x,y)=\begin{cases}
\ell(xy^{-1}), & s(x)=s(y)\\
\infty, & s(x)\not=s(y)
\end{cases}.
\]
On the other hand, given any invariant fiberwise extended metric $\rho$
on $\CG$, we associate a function 
\[
\ell_{\rho}:\CG\to[0,\infty),\quad g\mapsto\rho(g,s(g)),
\]
which does not take the value $\infty$ since $\rho$ is fiberwise. 

One may also wish length functions compatible  with the topology on the groupoid. We begin with the following definition.

\begin{defn}(\cite[Definition 4.4]{M-W})
	\label{def:coarse-length-function}Let $\ell:\CG\to[0,\infty)$ be
	a length function on an étale groupoid $\CG$. For any subset $K\subseteq\CG$,
	we write 
	\[
	\overline{\ell}(K)=\sup_{x\in K}\ell(x).
	\]
	We say $\ell$ is 
	\begin{itemize}
		\item \emph{coarse} if for any $K\subset\CG\setminus\CG^{(0)}$, one has  $\overline{\ell}(K)<\infty$ if and only if $K$ is precompact.
		\item \emph{continuous }if it is a continuous function with regard to the
		topology of $\CG$. 
	\end{itemize}
	Two length functions $\ell_{1},\ell_{2}$ are said to be \emph{coarsely
		equivalent} if for any $r>0$, we have 
	\[
	\sup\left\{ \ell_{1}(x):\ell_{2}(x)\leq r\right\} <\infty\quad\mbox{and}\quad\sup\left\{ \ell_{2}(x):\ell_{1}(x)\leq r\right\} <\infty.
	\]
\end{defn}

The following was established in \cite{M-W}.

\begin{thm}(\cite[Theorem 4.10, Remark 4.11]{M-W})
	\label{thm:coarse-length-functions}Up to coarse equivalence, any
	$\sigma$-compact groupoid $\CG$ has a unique coarse continuous length function $\ell$. If $\CG$ is also ample, the coarse length function $\ell$ can be chosen to be locally constant.
\end{thm} 

As what we have described above, the invariant fiberwise extended metric on a $\sigma$-compact groupoid $\CG$, induced by the unique-up-to-coarse-equivalence length function,  is called \textit{canonical} extended metric. Such an metric is denoted by $\rho_\CG$, or simply $\rho$. It has been also shown in \cite[Lemma 4.14]{M-W} that metric space $(\CG, \rho_\CG)$ is \textit{uniformly locally finite}\footnote{Some authors call this notion \textit{bounded geometry}} in the sense that for any $R>0$, there is a uniform finite upper bound on the cardinalities of all closed balls with radius $R$, namely, $\sup_{x\in \CG}|\bar{B}(x, R)|<\infty.$
It is fortunate to see that  the coarse structure  for groupoids looks same locally under the canonical metric.  This result was established in \cite[Lemma 5.10]{M-W} and is referred as the \emph{Local Slice Lemma}. The following result is a weak version of the Local Slice Lemma that we will actually use in this paper and it is established by the same construction. 


\begin{lem}\label{3.6}
Let $\CG$ be a $\sigma$-compact groupoid and let $\rho$ be a canonical extended metric on $\CG$ induced by a coarse continuous length function $\ell$. Then for any $u\in \GU$  such that $\CG^u_u=\{u\}$ and $R, \epsilon>0$, there is an $S\in [R, R+\epsilon)$ such that there is an open multisection $\CC=\{C_{i, j}: i, j\in I\}$ and an $i_u\in I$ such that
\begin{enumerate}
	\item $u\in C_{i_u, i_u}$ and
	\item $\bar{B}_\rho(v, S)\subset \bigsqcup_{i\in I}C_{i, i_u}v$ for any $v\in C_{i_u, i_u}$.
\end{enumerate} 
If $\CG$ is additionally ample, one may choose $\CC$ to be compact open and $\ell$ is constant on all $C_{i, j}\in \CC$.
\end{lem}
\begin{proof}
Let $u\in \GU$ such that $\CG^u_u=\{u\}$. First, since $(\CG, \rho)$ is locally finite,  the set $B_\rho(u, R+\epsilon)=\{x\in \CG_u: \ell(x)<R+\epsilon\}$ is finite and therefore \[
	\overline{\ell}\left(B_{\rho}(u,R+\eps)\right)=\max\left\{ \ell(x):x\in B_{\rho}(u,R+\eps)\right\} <R+\eps.
	\]
	Hence we may choose $S\in[R,R+\eps)\cap\left(\overline{\ell}\left(B_{\rho}(u,R+\eps)\right),R+\eps\right)$,
	e.g., 
	\[
	S=\max\left\{ R,\frac{\overline{\ell}\left(B_{\rho}(u,R+\eps)\right)+R+\eps}{2}\right\} ,
	\]
	which guarantees $\bar{B}_{\rho}(u,S)=B_{\rho}(u,R+\eps)$ and thus
	$S>\overline{\ell}\left(\bar{B}_{\rho}(u,S)\right)$. 
	
	Enumerate $\bar{B}_{\rho}(u,S)=\{x_0, x_1,\dots, x_n\}$ in which $x_0=u$. Then since $\CG^u_u=\{u\}$, choose an open bisection $U_i$ such that $x_i\in U_i$ for each $0\leq i\leq n$ and the family $\{r(U_i): 0\leq i\leq n\}$ is disjoint. Now define 
	\[
	L=\ell^{-1}([0,S])\setminus\left(\bigcup_{x\in\bar{B}_{\rho}(u,S)}U_{x}\right)\quad\mbox{and}\quad U=\CG^{(0)}\setminus s(L).
	\]
	Unpacking the definition and using the fact $\bar{B}_{\rho}(v,S)=\ell^{-1}([0,S])\cap s^{-1}(v)$
	for any $v\in\CG^{(0)}$, we have that
	\begin{align*}
	U=\{ v\in\CG^{(0)}:\bar{B}_{\rho}(v,S)\subseteq\bigcup_{x\in\bar{B}_{\rho}(u,S)}U_{x}\}
	\end{align*}
	is an open set. Define bisections $C_{i, 0}=U_iU$ and $C_{i, j}=C_{i, 0}C^{-1}_{j, 0}$ for $0\leq i, j\leq n$ and index set $I=\{0, \dots, n\}$ with $i_u=0$. Then by the definition of $U$, we have desired multisection $\CC$. 
	
	Finally, if $\CG$ is ample, one can shrink $U$ and all $U_i$ if necessary to obtain a compact open multisection $\CC$ such that $\ell$ is constant on all $C_{i, j}\in \CC$ by Theorem \ref{thm:coarse-length-functions}.
\end{proof}

\subsection{Fiberwise amenability}
 Amenability of extended uniformly locally finite spaces was first introduced by Block and Weinberger in \cite{B-W}.  Let $(X,d)$ be a locally finite extended metric space and $A$ be
 a subset of $X$. For any $R>0$ we define the following boundaries
 of $A$:
 \begin{enumerate}[label=(\roman*)]
 	\item \emph{outer $R$-boundary}: $\partial_{R}^{+}A=\{x\in X\setminus A:d(x,A)\leq R\}$;
 	\item \emph{inner $R$-boundary}: $\partial_{R}^{-}A=\{x\in A:d(x,X\setminus A)\leq R\}$; 
 	\item \emph{$R$-boundary}: $\partial_{R}A=\{x\in X:d(x,A)\leq R\ \textrm{and}\ d(x,X\setminus A)\leq R\}$. 
 \end{enumerate}

\begin{defn}
	\label{3.2} \label{def:metric-amenability}Let $(X,d)$ be a extended
	locally finite metric space. 
	\begin{enumerate}[label=(\roman*)]
		\item For $R>0$ and $\epsilon>0$, a finite non-empty set $F\subset X$
		is called $(R,\epsilon)$-F{ø}lner if it satisfies 
		\[
		\frac{|\partial_{R}F|}{|F|}\leq\epsilon.
		\]
		\item The space $(X,d)$ is called \textit{amenable} if, for every $R>0$
		and $\epsilon>0$, there exists an $(R,\epsilon)$-F{ø}lner set. 
	\end{enumerate}
\end{defn}

Note that a countable discrete group is amenable exactly when it is amenable as a metric space in the sense above. However, unlike groups,  a general metric space lacks homogeneity  so that amenability defined above sometimes is too weak to apply. So we need the following stronger version of amenability introduced in \cite{M-W}.

\begin{defn}(\cite[Defnition 3.5]{M-W})
	\label{3.0}\label{def:uniform-metric-amenability} An extended metric
	space $(X,d)$ is called \textit{ubiquitously amenable} (or \emph{ubiquitously
		metrically amenable}) if, for every $R>0$ and $\epsilon>0$, there
	exists an $S>0$ such that for any $x\in X$, there is an $(R,\epsilon)$-Følner
	set $F$ in the ball $\bar{B}(x,S)$. 
\end{defn}

It is not hard to see that a countable discrete group is amenable if and only if it is ubiquitously amenable because and right translation of a F{\o}lner set stays F{\o}lner. Analogously to metric setting, one may define (ubiquitously) amenability for groupoids. The following definitions are also borrowed from \cite{M-W}.

\begin{defn}(\cite[Defnition 5.1]{M-W})
	Let $\CG$ be a groupoid. For any subsets $A,K\subseteq\CG$, we define
	the following boundary sets:
	\begin{enumerate}[label=(\roman*)]
		\item \emph{left outer $K$-boundary}: $\partial_{K}^{+}A=(KA)\setminus A=\{yx\in\CG\setminus A:y\in K,x\in A\}$;
		\item \emph{left inner $K$-boundary}: $\partial_{K}^{-}A=A\cap(K^{-1}(\CG\setminus A))=\{x\in A:yx\in\CG\setminus A\mbox{ for some }y\in K\}$;
		\item \emph{left $K$-boundary}: $\partial_{K}A=\partial_{K}^{+}A\cup\partial_{K}^{-}A$. 
	\end{enumerate}
\end{defn}
Observe that if $A$ as above is contained in a single source fiber,
then $KA$ and all these boundary sets are also contained in this
source fiber. This is the reason for the terminology \emph{fiberwise
	amenability}. 
The following concept is analogous to the metric case, too. 
\begin{defn}(\cite[Defnition 5.3]{M-W})
	\label{def:groupoid-Folner} Let $\CG$ be a locally compact étale
	groupoid. For any subset $K\subseteq\CG$ and $\epsilon>0$, a finite
	non-empty set $F\subset X$ is called $(K,\epsilon)$-F{ø}lner if
	it satisfies 
	\[
	\frac{|\partial_{K}F|}{|F|}\leq\epsilon.
	\]
	\end{defn}
This leads to a natural definition of fiberwise amenability. 
\begin{defn}(\cite[Defnition 5.4]{M-W})
	\label{def:fiberwise-amenable}\label{4.1} Let $\CG$ be a locally
	compact étale groupoid. 
	\begin{enumerate}
		\item We say $\CG$ is \textit{fiberwise amenable} if for any compact subset
		$K$ of $\CG$ and any $\epsilon>0$, there exists a $(K,\epsilon)$-F{ø}lner
		set.
		\item We say $\CG$ is \textit{ubiquitously fiberwise amenable} if and only
		if for any compact subset $K$ of $\CG$ and any $\epsilon>0$, there
		exists a compact subset $L$ of $\CG$ such that for any unit $u\in\CG^{(0)}$,
		there is a $(K,\epsilon)$-F{ø}lner set in $Lu\cup\{u\}$. 
	\end{enumerate}
\end{defn}

On the other hand, we look at the metric space $(\CG, \rho_\CG)$ defined in subsection \ref{sub3.1}. The following theorem shows that (ubiquitously) fiberwise amenability (Definition \ref{def:fiberwise-amenable}) are equivalent to (ubiquitously) amenability of the metric space $(\CG, \rho_\CG)$.

\begin{prop}(\cite[Defnition 5.5]{M-W})
	\label{prop:fiberwise-amenability-metric} Let $\CG$ be a $\sigma$-compact
	locally compact Hausdorff étale groupoid and let $(\CG,\rho)$ be
	the extended metric space induced by a coarse length function $\ell$. 
	\begin{enumerate}
		\item The groupoid $\CG$ is fiberwise amenable if and only if for any compact
		subset $K$ of $\CG$ and any $\epsilon>0$, there exists a nonempty
		finite subset $F$ in $\CG$ satisfying 
		\[
		\frac{|KF|}{|F|}\leq1+\epsilon,
		\]
		if and only if $(\CG,\rho)$ is amenable in the sense of Definition
		\ref{3.2}.. 
		\item The groupoid $\CG$ is ubiquitously fiberwise amenable if and only
		if for any compact subset $K$ of $\CG$ and any $\epsilon>0$, there
		exists a compact subset $L$ of $\CG$ such that for any unit $u\in\CG^{(0)}$,
		there is a nonempty finite subset $F$ in $Lu\cup\{u\}$ satisfying
		\[
		\frac{|KF|}{|F|}\leq1+\epsilon,
		\]
		if and only if $(\CG,\rho)$ is ubiquitously amenable in the sense
		of Definition \ref{3.0}. 
	\end{enumerate}
\end{prop}

\begin{eg}\label{eg: amenability for actions}
	Let $\alpha:\Gamma\curvearrowright X$ be an action of a countable
	discrete group $\Gamma$ on a compact Hausdorff space $X$. We denote by $X\rtimes_{\alpha}\Gamma$
	the transformation groupoid of this action $\alpha$. When we equip
	$\Gamma$ with a proper length function $\ell_{\Gamma}$ and $X\rtimes_{\alpha}\Gamma$
	with the induced length function $\ell_{X\rtimes_{\alpha}\Gamma}:(\gamma x,\gamma,x)\mapsto\ell_{\Gamma}(\gamma)$
	each source
	fiber $(X\rtimes_{\alpha}\Gamma)_{x}=\{(\gamma x,\gamma,x):\gamma\in\Gamma\}$,
	for $x\in X$, becomes isometric to $\Gamma$. Therefore $X\rtimes_{\alpha}\Gamma$
	is ubiquitous fiberwise amenable if and only if $\Gamma$ is amenable. 
\end{eg}

As what we have mentioned above (see also \cite[Remark 5.7]{M-W}), ubiquitously fiberwise amenability is in general a more useful concept for us than fiberwise amenability. Nevertheless, we have the following theorem.

\begin{thm}(\cite[Theorem 5.13]{M-W})
	\label{4.01} \label{thm:minimal-fiberwise-amenability}Let $\CG$
	be a $\sigma$-compact groupoid. Suppose
	$\CG$ is minimal. Then $\CG$ is fiberwise amenable if and only
	if it is ubiquitously fiberwise amenable. \end{thm}

The following result, established in \cite{M-W} as well, is a groupoid version of ``amenability versus paradoxicality''.  Note that a slight difference in the statement of the first case here is that we use ubiquitously fiberwise amenability instead of fiberwise amenability in the original paper.  This is because we do not assume the minimality of the groupoid here. However, Theorem \ref{thm:minimal-fiberwise-amenability} implies that fiberwise amenability is equivalent to ubiquitously fiberwise amenability for minimal groupoids and thus yields a real dichotomy. See the original version of the theorem (\cite[Theorem 5.15]{M-W}).

\begin{thm}(\cite[Theorem 5.15]{M-W})
	\label{thm: dichotomy} Let $\CG$ be a $\sigma$-compact groupoid. Equip $\CG$ with the canonical extended metric $\rho$. Then we have the following.
	\begin{enumerate}
		\item If $\CG$ is ubiquitously fiberwise amenable then for all $R,\epsilon>0$ there
		is a compact set $K$ with $\CG^{(0)}\subset K\subset\CG$ such that
		for all compact set $L\subset\CG$ and all unit $u\in G^{(0)}$ there
		is a finite set $F_{u}$ satisfying 
		\[
		Lu\subset F_{u}\subset KLu\ \textrm{and}\ \bar{B}_{\rho}(F_{u},R)\leq(1+\epsilon)|F_{u}|.
		\]
		
		\item If $\CG$ is not fiberwise amenable then for all compact set $L\subset\CG$
		and $n\in\mathbb{N}$ there is a compact set $K\subset\CG$ such that
		for all compact set $M\subset\CG$ and all $u\in G^{(0)}$, the set
		$KMu$ contains at least $n|Mu|$ many disjoint sets of the form $L\gamma u$,
		i.e., there exists a disjoint family $\{L\gamma_{i}u\subset KMu:i=1,\dots,n|Mu|\}$. 
	\end{enumerate}
\end{thm}

From now on, we only consider \textbf{$\sigma$-compact, locally
	compact, Hausdorff, étale  ample topological groupoids whose unit spaces are compact}. 

\section{F{ø}lner sequences and Banach densities for ample groupoids}
 In this section, we introduce concepts of F{\o}lner sequences for groupoids and Banach densities for groupoids that admits F{\o}lner sequences.
 
 \subsection{F{\o}lner sequences for groupoids}
\begin{defn}\label{defn_normal folner set}
Let $\CG$ be an ample groupoid. Let $K$ be a compact set in $\CG$, $\epsilon>0$ and $N\in \N_+$, we say a compact open set $S$ in $\CG$ is a \textit{normal $(K, \epsilon)$-F{\o}lner set}  if there is a family $\{C^l_{i, j}: i, j\in F_l, 1\leq l\leq m\}$ of compact open bisections with an $i_l\in F_l$ for each $1\leq l\leq m$ such that 
\begin{enumerate}
	\item $\CC_l=\{C^l_{i, j}: i, j\in F_l\}$ is a compact open multisection for each $1\leq l\leq m$;
	\item  $S=\bigcup_{l=1}^m\bigsqcup_{i\in F_l}C^l_{i, i_l}$;
	\item $\GU=\bigsqcup_{l=1}^mC^l_{i_l, i_l}$;  
	\item $Su$ is $(K, \epsilon)$-F{\o}lner for any $u\in \GU$ in the sense of Definition \ref{def:groupoid-Folner}.
\end{enumerate}
\end{defn}

The following lemma will be used to construct normal F{\o}lner sets in groupoids.

\begin{lem}\label{lem: folner slice}
Let $\CG$ be a ubiquitous fiberwise amenable groupoid and $u\in \GU$ such that $\CG^u_u=\{u\}$. Let $K$ be a compact set in $\CG$ and $\epsilon>0$. Then there is a compact open multisection $\CC=\{C_{i, j}: i, j\in F\}$ and an $i_u\in F$ such that 
\begin{enumerate}
	\item $u\in C_{i_u, i_u}$ and
	\item if we write $D=\bigsqcup_{i\in F}C_{i, i_u}$ then $Dv$ is $(K, \epsilon)$-F{\o}lner for any $v\in C_{i_u, i_u}$.
\end{enumerate}
\end{lem}
\begin{proof}
Equipped with $\CG$ the coarse continuous length function $\ell$ and denote by $\rho$ the induced canonical invariant metric. Since $\CG$ is ample, without loss of any generality, one may assume $\ell$ is locally constant by Theorem \ref{thm:coarse-length-functions}.

Let $K, \epsilon$ be given. Without loss of any generality, one may assume $\GU\subset K$. Now, let $R>0$ such that $KF\subset \bar{B}(F, R)$ for any $F\subset \CG_w$ and any $w\in \GU$. Then for this $R$, Theorem \ref{thm: dichotomy} implies that there is a compact set $M$ with $\GU\subset M\subset \CG$ such that for any compact set $T$ in $\CG$ and any $w\in \GU$ there is a finite set $E_w$ such that \[
Tu\subset E_{w}\subset MTw\ \textrm{and}\ \bar{B}_{\rho}(E_{w},R)\leq(1+\epsilon)|E_{w}|.
\]
Now let $T=\GU$ and choose $R_1>0$ such that
\[E_w\subset Mw\subset \bar{B}_{\rho}(w,R_1)\subset \bar{B}_{\rho}(w,R+R_1)\]
for any $w\in \GU$.  Now for our $u\in \GU$ in the assumption, $R+R_1$ and $\epsilon>0$, Lemma \ref{3.6} shows that there are an $N\in [R+R_1, R+R_1+\epsilon)$, a compact open multisection $\CT=\{C_{i, j}: i, j\in T_u\}$ and an $i_u\in T$ with $u\in C_{i_u, i_u}$ such that
\[\bar{B}_\rho(v, N)\subset\bigsqcup_{i\in T_u}C^u_{i, i_u}v\]
for any $v\in C^u_{i_u, i_u}$ and $\ell$ is constant on each $C^u_{i, j}$ for each  $i, j\in T_u$. Now since $N>R+R_1$, one has
\[E_{u}\subset Mu\subset\bar{B}_{\rho}(u,R_1)\subset \bar{B}_{\rho}(u,R+R_1)\subset \bigsqcup_{i\in T}C^l_{i, i_u}u\]
and thus there is a $F\subset T$ such that 
\[E_{u}=\bigsqcup_{i\in F}C_{i, i_u}u.\]
Recall that $|\bar{B}_{\rho}(E_{u},R)|\leq(1+\epsilon)|E_{u}|$. We write $D=\bigsqcup_{i\in F}C_{i, i_u}$ for simplicity and observe that $E_u=Du$ actually.

Now fix a $v\in C_{i_u, i_u}$, we define $f: \bigsqcup_{i\in T}C_{i, i_u}v\to \bigsqcup_{i\in T}C_{i, i_u}u$ by claiming that if $C_{j, i_u}v=\{x\}$ then $\{f(x)\}=C_{j, i_u}u$. Since all $C_{j, i_u}$ are bisections, $f$ is bijective. Furthermore, observe that $Dv=\bigsqcup_{i\in F}C_{i, i_u}v$ and thus one has $f(Dv)=E_{u}$. Now if $x\in \bar{B}_\rho(Dv, R)$ then  there is an $j\in F$ and a $y\in C_{j, i_u}v$ such that $\rho(x, y)=r\leq R$. Now since $\ell$ is constant on the bisection $C^l_{i, i_l}$, one has $\ell(y)=\ell(z)\leq R_1$ where $\{z\}=C^l_{i, i_u}u$ since $C^l_{i, i_u}u\subset E_{u}\subset \bar{B}_\rho(u, R_1)$. Thus $\ell(x)\leq \ell(y)+\rho(x, y)=R_1+R$, which implies that $x\in \bar{B}_\rho(v, R_1+R)\subset \bigsqcup_{i\in T}C_{i, i_u}v.$ 

So let $\{x\}=C^l_{j_1, i_l}v$ for some $j_1\in T_l$. Now for the $y\in Sv$ above with $\{y\}=C^l_{j, i_l}v$ for some $j\in F_l$ such that $\rho(x, y)=r\leq R_1$, one has $yx^{-1}\in C^l_{j, j_1}$ and  $\ell(yx^{-1})=r\leq R$. Now since $f(y)f(x)^{-1}\in C^l_{j, j_1}$ as well and $\ell$ is constant on $C^l_{j, j_1}$, one has $f(x)\in \bar{B}_\rho(E_{u}, R_1)$ because $f(y)\in E_{u}$. Therefore we have verified that $f(\bar{B}_\rho(Dv, R))\subset \bar{B}_\rho(E_{u}, R)$. Then because $f$ is bijective and $f(Dv)=E_{u}$, one has 
\[|\bar{B}_\rho(Dv, R)|\leq |\bar{B}_\rho(E_{u}, R)|\leq (1+\epsilon)|E_{u}|=(1+\epsilon)|Dv|.\] Now for any  $v\in C^l_{i_u, i_u}$ one has
\[|KDv|\leq |\bar{B}_\rho(Dv, R)|\leq (1+\epsilon)|Dv|,\]
which implies that $Dv$ is $(K, \epsilon)$-F{\o}lner.
\end{proof}

Then using Lemma \ref{lem: folner slice}, one can show normal F{\o}lner sets always exist in the following natural cases.

\begin{prop}\label{prop:normal folner set}
Let $\CG$ be a principal  ubiquitous fiberwise amenable groupoid. Then for each compact set $K$ in $\CG$ and $\epsilon>0$ there is a normal  $(K, \epsilon)$-F{\o}lner set $S$ in $\CG$.
\end{prop}
\begin{proof}
Since $\CG$ is principal, Lemma \ref{lem: folner slice} implies that for any $u\in \GU$, there is a compact open multisection $\CC_u=\{C^u_{i, j}: i, j\in F_u\}$ and an $i_u\in F_u$ such that 
\begin{enumerate}
	\item $u\in C^u_{i_u, i_u}$ and
	\item if we write $D_u=\bigsqcup_{i\in F_u}C^u_{i, i_u}$ then $D_uv$ is $(K, \epsilon)$-F{\o}lner for any $v\in C^u_{i_u, i_u}$.
\end{enumerate}
Then using $\GU$ is compact and shrink all $C^u_{i, j}$ if necessary, one has a family $\{\CC_1, \dots, \CC_m\}$  of compact open multisections such that 
\begin{enumerate}
	\item each $\CC_l=\{C^l_{i, j}: i, j\in F_l\}$;
	\item for each $l$ there is an $i_l\in F_l$ such that  $\GU=\bigsqcup_{l=1}^mC^l_{i_l, i_l}$ and
	\item  if we write $D_l=\bigsqcup_{i\in F_l}C^l_{i, i_l}$ then $D_lv$ is $(K, \epsilon)$-F{\o}lner for any $v\in C^l_{i_l, i_l}$
	\end{enumerate}
Now we define $S=\bigcup_{l=1}^mD_l=\bigcup_{l=1}^m\bigsqcup_{i\in F_l}C^l_{i, i_l}$, which is a normal $(K, \epsilon)$-F{\o}lner set by our construction. 
\end{proof}

When the groupoid $\CG$ is minimal, we may largely weaken the conditions needed in  Proposition \ref{prop:normal folner set}. We first introduce the following concepts.

\begin{defn}
	Let $\CG$ be an ample groupoid and $K\subset \CG$ a compact set. We write $\|K\|=\sup_{u\in s(K)}|Ku|$.
\end{defn}

\begin{rmk}\label{5.2}
Note that the number $\|K\|$ is always a finite integer since $K$ is compact. In addition, let $u\in \GU$ and $N\subset \CG_u$ a finite set. Then one has $|KN|=|\bigcup_{x\in N}Kx|\leq \sum_{x\in N}|Kx|\leq \|K\||N|$.
\end{rmk}

Like in the case of amenable groups, right shifts or small perturbations of a  F{\o}lner set stays F{\o}lner.

\begin{lem}\label{lem: shrink folner set}
	Let $\CG$ be a groupoid and $u\in \GU$. Suppose $F\subset \CG_u$ is a $(K, \delta)$-F{\o}lner set and  $\gamma\in \CG$ such that $r(\gamma)=u$. Then
	\begin{enumerate}
		\item if $F'\subset F$ such that $|F'|\geq (1-\epsilon)|F|$ then $F'\gamma$ is $(K, \delta+\epsilon)$-F{\o}lner.
		\item if $|F'\Delta F|<\epsilon|F|$, then $F'$ is $(K, \frac{\epsilon(2+\|K\|)}{1-\epsilon})$-F{\o}lner.
	\end{enumerate}
	 
\end{lem}
\begin{proof}
	For the claim (1), simply observe that 
	\[KF'\gamma\setminus F'\gamma\subset KF\gamma\setminus F'\gamma=(KF\setminus F')\gamma\subset ((KF\setminus F)\cup (F\setminus F'))\gamma.\]
	This implies that $|KF'\gamma\setminus F'\gamma|\leq |KF\setminus F|+|F\setminus F'|=\delta+\epsilon$.
	
	For the  claim (2), observe that 
	\begin{align*}
		|KF'|\leq |K(F'\cap F)|+|K (F'\setminus F)|\leq (1+\epsilon)|F|+(\|K\|\epsilon)|F|\leq \frac{1+(\|K\|+1)\epsilon}{1-\epsilon}|F'|
	\end{align*}
	and thus $|KF'|\leq (1+\frac{\epsilon(\|K\|+2)}{1-\epsilon})|F'|$.
\end{proof}

 We then quote the following result.

\begin{lem}\cite[Lemma 7.5]{M-W}\label{lem: big ball}	Let $\CG$ be a minimal groupoid equipped with the canonical metric $\rho$. Then for any $N>0$ there is a $R>0$ such that for any $x\in \CG$ one has $|\bar{B}_\rho(x, R)|>N$.
\end{lem}

Then the next lemma shows that in minimal groupoids, F{\o}lner sets can be chosen arbitrary large.

\begin{lem}\label{lem: Folner length arbitrary long}	Let $\CG$ be a  minimal fiberwise amenable groupoid, $n\in \N^+$ and $\epsilon>0$. Suppose $u\in \GU$ such that $\CG^u_u=\{u\}$. Then  there is a compact set $K$ such that for any $(K, \epsilon)$-F{\o}lner set $F\subset \CG_u$ one has $|F|>n$.
\end{lem}
\begin{proof}
	By Lemma \ref{lem: big ball}, for the $n\in \N_+$, choose $R>0$ such that for any $x\in \CG$ one has $|\bar{B}_\rho(x, R)|>n(1+\epsilon)$, where $\rho$ is the canonical extended metric on $\CG$. Now choose a compact set $K$ in $\CG$ such that $\ell^{-1}([0, R])\subset K$. 
	Now our choice of $K$ implies that   $\bar{B}_\rho(F, R)\subset KF$ and thus 
	\[n(1+\epsilon)< |\bar{B}_\rho(F, R)|\leq |KF|\leq (1+\epsilon)|F|.\]	This implies that $|F|>n$.
\end{proof}

Now we are ready to prove the following result.

\begin{prop}\label{prop: minimal groupoid normal folner}
	Let $\CG$ be a minimal fiberwise amenable groupoid and there exists a $u\in \GU$ such that $\CG^u_u=\{u\}$. Then for any compact set $K$ in $\CG$ and $\epsilon>0$, there is a $(K, \epsilon)$-F{\o}lner normal set $S$ in $\CG$.
\end{prop}
\begin{proof}
First, enlarge $K$ if necessary, Lemma \ref{lem: Folner length arbitrary long} implies that $|F|>2/\epsilon$ for any $(K, \epsilon)$-F{\o}lner set $F$. Now let $0<\delta<\epsilon$ be small enough such that if $F$ is $(K, \delta)$-F{\o}lner then $F'$ is $(K, \epsilon)$-F{\o}lner whenever $|F\Delta F'|=2<\epsilon|F|$. 
Then \ref{thm:minimal-fiberwise-amenability} shows that $\CG$ is actually ubiquitous fiberwise amenable and thus Lemma \ref{lem: folner slice} implies that there is a compact open multisection $\CC=\{C_{i, j}: i, j\in F\}$ and an $i_u\in F$ such that 
\begin{enumerate}
	\item $u\in C_{i_u, i_u}$ and
	\item if we write $D=\bigsqcup_{i\in F}C_{i, i_u}$ then $Dv$ is $(K, \delta)$-F{\o}lner for any $v\in C_{i_u, i_u}$.
\end{enumerate}
Without loss of generality, we write $F=\{1, \dots, n\}$ for some $n\in \N^+$ and identify $i_u=0$.
Now using the fact that $\CG$ is minimal,  and $\GU\setminus C_{1, 1}$ is still a compact set, there is a family of compact open bisections $\{B_1,\dots, B_m\}$ such that $\{s(B_1),\dots, s(B_m)\}$ form a clopen partition of $\GU\setminus C_{1, 1}$ and $r(B_l)\subset C_{1, 1}$ for all $1\leq l\leq m$.
Now for any $1\leq l\leq m$, define $C^l_{0, 0}=s(B_l)$ and $C^l_{i, 0}=C_{i, 0}B_l$ for all $1\leq i\leq n-1$. Then, define $C^l_{i, j}=C^l_{i, 0}(C^l_{j, 0})^{-1}$ for any $0\leq i, j\leq n-1$. We write $F_l=\{0, \dots, n-1\}$ and observe that $\{C^l_{i, j}: i, j\in F_l\}$ is a multisection.

Finally, write $D_l=\bigsqcup_{i\in F_l}C^l_{i, 0}$. Then for any $w\in C^l_{0, 0}=s(B_l)$, let $\gamma\in B_l$ with $s(\gamma)=w$ and observe that $|D_lw\Delta D\gamma|=2$. This implies that $D_lw$ is $(K, \epsilon)$-F{\o}lner because $D\gamma$, as a right shift of $Dr(\gamma)$, is $(K, \delta)$-F{\o}lner.
Then, by our construction, it is not hard to see \[S=(\bigcup_{l=1}^m\bigsqcup_{i\in F_l}C^l_{i, 0})\cup (\bigsqcup_{j\in F}C_{j, 0})\]
is a normal $(K, \epsilon)$-F{\o}lner set.
\end{proof}

Another useful family of groupoids are \emph{almost finite} groupoids. They were originally introduced by Matui in \cite{Matui} to study Homology theory and topological full groups of groupoids. We record its definition as follows.

\begin{defn}\cite[Definition 6.2]{Matui}\label{defn: almost finiteness}
	Let $\CG$ be a
	groupoid with a compact unit space. $\CG$ is called almost
	finite if for any compact set $K$ in $\CG$ and $\epsilon>0$ there
	is a compact open elementary subgroupoid $\CH$ of $\CG$ with $\HU=\GU$
	such that 
	$|K\CH u\setminus\CH u|<\epsilon|\CH u|$  for any $u\in\GU$.
\end{defn}

\begin{prop}\label{prop: almost finite folner sequence}
	Let $\CG$ be an almost finite groupoid. For any compact set $K$ in $\CG$ and $\epsilon>0$, there is a normal $(K, \epsilon)$-F{\o}lner set.
\end{prop}
\begin{proof}
	Let $K\subset \CG$ be a compact set and $\epsilon>0$. Since $\CG$ is almost finite, there is a compact open elementary subgroupoid $\CH$ with $\HU=\GU$
	such that 
	$|K\CH u\setminus\CH u|<\epsilon|\CH u|$  for any $u\in\GU$. We claim that $\CH$ itself is a normal $(K, \epsilon)$-F{\o}lner set. Indeed, first, by considering the fundamental domain of $\CH$, there is a compact open castle $\CC=\{C_{i,j}^{l}:i,j\in F_{l},l\in I\}$ such that $\CH=\bigcup\CC=\bigsqcup_{l\in I}\bigsqcup_{i,j\in F_{l}}C_{i,j}^{l}$. Define a new index set $J=\{p=(l, i): i\in F_l, l\in I\}$ and for each $p\in J$ define $E_p=F_l$ and a particular index $d_p\in E_p$ by $d_p=i$. Using these new notations, observe that $\CH=\bigsqcup_{p\in J}\bigsqcup_{j\in E_p}C^l_{j, d_p}$, which satisfies Definition \ref{defn_normal folner set}.
\end{proof}

Now we announce the following definition, which is an analogue of the concept of F{\o}lner sequences in countable discrete groups.
\begin{defn}\label{defn: folner sequence for groupoids}
Let $\CG$ be a  groupoid. We say a sequence of normal F{\o}lner sets $\{S_n: n\in \N\}$ is  a F{\o}lner sequence in $\CG$ if for any compact set $K$ in $\CG$ and any $u\in \GU$ one has 
\[\frac{|KS_nu\setminus S_nu|}{|S_nu|}\to 0.\]
\end{defn}

In the group action case, if the acting group is amenable then a F{\o}lner sequence of it naturally induces a F{\o}lner sequence of the transformation groupoid of the action.

\begin{eg}\label{eg: canonical folner sequence}
For the transformation groupoid $X\rtimes\Gamma$ case in which $\Gamma$ is a countable discrete amenable group and $X$ is the Cantor set. Suppose $K\subset \Gamma$ is a finite set and $\epsilon>0$. Let  $F$ be a $(K, \epsilon)$-F{\o}lner set containing $e_\Gamma$, and choose a clopen partition $\{V_1,\dots, V_m\}$ of $X$ such that $(F, V_l)$ are \emph{towers} for all $l\leq m$ in the sense that $\{sV_l: s\in F\}$ is disjoint. Now for $s, t\in F$ define a bisection $C^l_{st}=\{(st^{-1}x, st^{-1}, x): x\in tV_l\}$. This implies that  $\CC_l=\{C^l_{st}: s, t\in F\}$ is a multisection for each $1\leq l\leq m$. Observe that $S=\{(\gamma x, \gamma, x): x\in X, \gamma\in F\}=\bigcup_{l=1}^m\bigsqcup_{s\in F}C^l_{s, e_\Gamma}$ is a compact open normal F{\o}lner set. Now one can simply define $S_{n, \Gamma, X}=\{(\gamma x, \gamma, x): x\in X, \gamma\in F_n\}$ where  $\{F_n\subset \Gamma: n\in \N\}$ is a usual F{\o}lner sequence of $\Gamma$ such that $e_\Gamma\in F_n$ for each $n\in \N$. Therefore, all $S_{n, \Gamma, X}$ forms a F{\o}lner sequence of $X\rtimes\Gamma$. We remark that each $S_{n, \Gamma, X}$ has the homogeneity in the sense that each partial source fiber $S_nx$ has the same geometric structure with $F_n$ (see Example \ref{eg: amenability for actions}).
\end{eg}

However, in general groupoids case, because of its fiberwise nature, one cannot expect that such a geometric homogeneity emerges in F{\o}lner sequences. Nevertheless, the normal F{\o}lner sets yields the local homogeneity.
\begin{prop}\label{prop: exist folner sequence}
Let $\CG$ be a groupoid such that for any compact $K$ in $\CG$ and $\epsilon>0$ there is a normal $(K, \epsilon)$-F{\o}lner set (this includes cases in Propositions \ref{prop:normal folner set}, \ref{prop: minimal groupoid normal folner} and \ref{prop: almost finite folner sequence}). There exists a F{\o}lner sequence $\{S_n: n\in \N\}$ for $\CG$.
\end{prop}
\begin{proof}
Let $\{K_n\}$ be an increasing sequence of compact open sets in $\CG$ such that $\bigcup_{n=0}^\infty K_n=\CG$ with $K_0=\GU$. Let $\{\epsilon_n: n\in \N\}$ be a decreasing sequence of positive numbers such that $\lim_{n\to \infty}\epsilon_n=0$. Then the assumption implies that  for each $n$ there is a normal $(K_n, \epsilon_n)$-F{\o}lner set $S_n$ such that $|K_nS_nu|\leq (1+\epsilon_n)|S_nu|$ for each $u\in \GU$. Now let $K$ be a compact set in $\CG$. Because each $K_n$ is compact open, there is an $N$ such that $K\subset K_n$ for any $n>N$. Then for each $u\in \GU$ one has 
\[|KS_n\setminus S_nu|\leq |K_nS_nu\setminus S_nu|\leq \epsilon_n|S_nu|.\]
This thus implies that $\{S_n: n\in \N\}$ is a F{\o}lner sequence.
\end{proof}

\subsection{Banach densities for groupoids}
The classical Ornstein-Weiss quasi-tiling theorem says that the underlying space for a free p.m.p action of an amenable group can be decomposed into towers with F{\o}lner shapes modulo an arbitrary small set with respect to the invariant measure (see e.g. \cite{Kerr-L}). This result has been generalized in \cite{C-J-K-M-S-T} and \cite{D-G} in p.m.p and topological setting, respectively, by considering Banach densities so that one can deal with many invariant measures thanks to the good relation between Banach densities and invariant measures (see Proposition \ref{prop:density vs measure} below). One of main goals in this paper is to prove a topological groupoid version of Ornstein-Weiss quasi tilling theorem, which means one can always find a compact open castle with arbitrary good fiberwise F{\o}lner shape, that covers almost all of the unit space. Such a good property that admits an arbitrary ``good'' quasi-tilling in this paper is called \emph{almost finiteness in measure}.

\begin{defn}\label{defn: af in measure}
	Let $\CG$ be a groupoid. We say $\CG$ is almost finite in measure if for any compact set $K\subset \CG$ and $\epsilon>0$ there is a compact open elementary subgroupoid $\CH$ such that
	\begin{enumerate}
		\item $|K\CH u\setminus \CH u|<\epsilon|\CH u|$ for any $u\in \HU$ and
		\item $\mu(\HU)>1-\epsilon$ for any $\mu\in M(\CG)$.
	\end{enumerate} 
\end{defn}
Suppose $\CG$ is a transformation groupoid of an action of an amenable group on the Cantor set, the above definition coincides with the notion with the same name introduced in \cite{D-G}. This validate the name of our notion. In addition, our almost finiteness in measure is a weaker version of Matui's almost finiteness recorded in Definition \ref{defn: almost finiteness}.

Following the strategy in \cite{C-J-K-M-S-T} and \cite{D-G}, the first step for us is to develop a theory of Banach densities for (ubiquitous) fiberwise amenable groupoids $\CG$.  The proposed density should have good relation with $\CG$-invariant probability measures as in group action case. See Proposition \ref{prop:density vs measure} below.  However, due to lack of the global homogeneity, a simple generalization of original Banach densities to the groupoid setting may be ill-behaved as we will see in Remark \ref{rem: bad density}. We first recall the definition of Banach densities for group actions on compact Hausdorff spaces (see \cite{D-G} for example) for comparison.

\begin{defn}\label{defn: density group action}
Let $\Gamma$ be a countable discrete group and $X$ a compact Hausdorff space. Let $\alpha: \Gamma\curvearrowright X$ be a continuous action.  For a set $A\subset X$ we define
lower and upper Banach densities of $A$ (with respect to the action $\alpha$) by 
\[\barbelow{D}(A)=\sup_{\substack{F\subset \Gamma \\ F\ \text{finite}}}\inf_{x\in X}\frac{1}{|F|}\sum_{s\in F}1_A(sx)\ \ \ \text{and}\ \ \ \bar{D}(A)=\inf_{\substack{F\subset \Gamma \\ F\ \text{finite}}}\sup_{x\in X}\frac{1}{|F|}\sum_{s\in F}1_A(sx).\]
\end{defn}

\begin{prop}(\cite[Proposition 3.3]{D-G})\label{prop:density vs measure}
	Let $\alpha: \Gamma\curvearrowright X$ be a continuous action of a countable discrete group $\Gamma$ on a compact Hausdorff space. Let $A$ be a closed set in $X$. Then $\bar{D}(A)\geq \sup_{\mu\in M_\Gamma(X)}\mu(A)$. If $\Gamma$ is amenable, one actually has $\bar{D}(A)=\sup_{\mu\in M_\Gamma(X)}\mu(A)$.
\end{prop}

Motivated by Definition \ref{defn: density group action}, it is natural to try the following definition for groupoids, i.e.,
\[\barbelow{D}_\CG(A)=\sup_{\substack{F\subset \CG \\ F\ \text{compact}}}\inf_{u\in \GU}\frac{1}{|Fu|}\sum_{\gamma\in F}1_A(r(\gamma))\] and
\[ \bar{D}_\CG(A)=\inf_{\substack{F\subset \CG \\ F\ \text{compact}}}\sup_{u\in \GU}\frac{1}{|Fu|}\sum_{\gamma\in F}1_A(r(\gamma)).\]
However, as we will see in Remark \ref{rem: bad density} by using the following example, the fiberwise nature of groupoids makes the proposed definition of groupoid Banach densities above ill-behaved so that there is no good relation to invariant measures even for transformation groupoids.

\begin{eg}\label{eg: ill-behaved density}
	Let $\alpha:\Z\curvearrowright X$ be a minimal action where $X$ is the Cantor set. Fix a small clopen set $V$ and a decreasing sequence $V=V_0\supset V_1\supset\cdots$ of clopen sets such that $\{V_n, 1\cdot V_n, \dots, n\cdot V_n\}$ are disjoint. For each $V_n$, by minimality, there are a clopen partition $\CP_n=\{O^n_1,\dots, O^n_{m_n}\}$ of $X$ finer than the compact open partition $\{X\setminus \bigsqcup_{k=0}^n k\cdot V_n, V_n, 1\cdot V_n, \dots, n\cdot V_n \}$ and $s^n_1, \dots, s^n_{m_n}\in \Z$ such that $s^n_kO^n_k\subset V_n$ and if  $O^n_l\subset k\cdot V_n$ then $s^n_l\neq -k$ for any $0\leq k\leq n$ and $1\leq l\leq m_n$.
	
	Now for each $n$ write $F_n=\{0,\dots, n\}$ and define $T_n=\{(\gamma x, \gamma, x):\gamma\in \{0\}\cup s^n_{k}+F_n, x\in O^n_k, 1\leq k\leq m_n\}$. Then $\{T_n: n\in \N\}$ is a clopen F{\o}lner sequence satisfying Definition \ref{defn: folner sequence for groupoids}. 
\end{eg}

\begin{rmk}\label{rem: bad density}
Let $\CG=X\rtimes \Z$ as in Example \ref{eg: ill-behaved density}.  One may make the original clopen $V=V_0$ small enough such that $\sup_{\mu\in M_\Z(X)}\mu(V)<1/3$ and thus its complement $A=X\setminus V_1$ satisfies $\inf_{\mu\in M_\Z(X)}\mu(A)>2/3$. Then for $T_1$ and any $x\in X$, one always has $\sum_{\gamma\in T_1x}1_A(r(\gamma))\leq 2$ while $|T_1x|=3$, which implies that \[\bar{D}_\CG(A)=\inf_{\substack{F\subset \CG \\ F\ \text{compact}}}\sup_{u\in \GU}\frac{1}{|Fu|}\sum_{\gamma\in F}1_A(r(\gamma))\leq 2/3<\inf_{\mu\in M_\Z(X)}\mu(A).\]
One may want to revise the definition of proposed densities above by replacing the compact set $F$ by compact open bisections $B$ with $s(B)=r(B)=\GU$, which essentially are the Banach densities for the natural action of the topological full group $[[\CG]]$ on $\GU$. This indeed will provide good relation to $\CG$-invariant measures by Proposition \ref{prop:density vs measure}. But it is unknown how to make normal F{\o}lner sets in the groupoid get involved in the work when $\CG$ is ubiquitous fiberwise amenable.
\end{rmk}

On the other hand, in the case of a dynamical system $\alpha: \Gamma \curvearrowright X$ of amenable acting group $\Gamma$ with a F{\o}lner sequence $\{F_n: n\in \N\}$ on the compact Hausdorff space $X$,  for densities defined in Definition \ref{defn: density group action} and any $A\subset X$ one actually has 
\[\bar{D}(A)=\lim_{n\to\infty}\sup_{x\in X}\frac{1}{|F_n|}\sum_{s\in F_n}1_A(sx)\ \ \text{and}\ \ \barbelow{D}(A)=\lim_{n\to\infty}\inf_{x\in X}\frac{1}{|F_n|}\sum_{s\in F_n}1_A(sx).\]
See e.g. \cite{D-G}. Motivated by this, we naturally have the following definition of groupoid densities.

\begin{defn}\label{defn: densities for groupoid} Let $\CG$ be a groupoid and $\CS=\{S_n: n\in\N\}$  a F{\o}lner sequence of $\CG$, we define the following densities for any set $A$ in $\GU$ with respect to $\CS$.
	\[\barbelow{D}^+_{\CS}(A)=\limsup_{n\to \infty}\inf_{u\in \GU}\frac{1}{|S_nu|}\sum_{\gamma\in S_nu}1_A(r(\gamma))\]
	\[\barbelow{D}^-_{\CS}(A)=\liminf_{n\to \infty}\inf_{u\in \GU}\frac{1}{|S_nu|}\sum_{\gamma\in S_nu}1_A(r(\gamma))\]
	\[\bar{D}^+_{\CS}(A)=\limsup_{n\to \infty}\sup_{u\in \GU}\frac{1}{|S_nu|}\sum_{\gamma\in S_nu}1_A(r(\gamma)).\]
	\[\bar{D}^-_{\CS}(A)=\liminf_{n\to \infty}\sup_{u\in \GU}\frac{1}{|S_nu|}\sum_{\gamma\in S_nu}1_A(r(\gamma)).\]
\end{defn}

\begin{rmk}
We warn that the densities defined above depends on the choice on the F{\o}lner sequence $\CS$. Then, by definition, observe $\bar{D}^+_{\CS}(A)=1-\barbelow{D}^-_{\CS}(\GU\setminus A)$ and $\bar{D}^-_{\CS}(A)=1-\barbelow{D}^+_{\CS}(\GU\setminus A)$. Finally, the definition of normal F{\o}lner set implies that $r|_{S_nu}$ is injective for any $n\in \N$ and $u\in \GU$. This implies that $|A\cap r(S_nu)|=\sum_{\gamma\in S_nu}1_A(r(\gamma))$, which will be used frequently without no further mention.
\end{rmk}

The following is our first result on the relation between our densities and $\CG$- invariant measures. We denote by $M(\GU)$ the set of all Borel probability measures on $\GU$ and by $M(\CG)$ the set of all $\CG$-invariant Borel probability measures on $\GU$.

\begin{lem}\label{lem: density less measure}
Let $\CG$ be a groupoid that admits a F{\o}lner sequence $\CS=\{S_n: n\in \N\}$ and $A$ a compact open set in $\GU$. Then there is a $\nu\in M(\CG)$ such that $\bar{D}^+_\CS(A)=\nu(A)\leq \sup_{\mu\in M(\CG)}\mu(A)$.
\end{lem}
\begin{proof}
For any $n\in \N$, define
	\[\bar{D}^+_{S_n}(A)=\sup_{u\in \GU}\frac{1}{|S_nu|}\sum_{\gamma\in S_nu}1_A(r(\gamma)).\]
Then by definition of $\bar{D}^+_\CS(A)$, there is a subsequence $\bar{D}^+_{S_{n_k}}(A)$ converging to $\bar{D}^+_\CS(A)$. For each $k$, choose a $u_k\in \GU$ such that
\[|\bar{D}_{S_{n_k}}(A)-\frac{1}{|S_{n_k}u_k|}\sum_{\gamma\in S_{n_k}u_k}1_A(r(\gamma))|<\frac{1}{k}.\]
This actually implies that
\[\bar{D}^+_\CS(A)=\lim_{k\to \infty}\frac{1}{|S_{n_k}u_k|}\sum_{\gamma\in S_{n_k}u_k}1_A(r(\gamma)).\]
Now define a probability measure $\mu_k=(1/|S_{n_k}u_k|)\sum_{\gamma\in S_{n_k}u_k}\delta_{r(\gamma)}$ on $\GU$ and let $\nu$ be a cluster point of the sequence $\{\mu_k\}$ in $M(\GU)$ under weak*-topology. Then by passing to a subsequence if necessary one has $\nu(A)=\lim_{k\to \infty}\mu_k(A)$. Now it is left to verify that $\nu\in M(\CG)$. But this follows  by the same argument of \cite[Proposition 5.9]{M-W}.
\end{proof}

On the other hand, let  $\alpha: \Gamma\curvearrowright X$ be an amenable group $\Gamma$ acting on the Cantor set $X$.  Applying the method in the p.m.p setting to Cantor dynamical systems, the way to establish Ornstein-Weiss quasi-tiling, or almost finiteness in measure in \cite{D-G}, is to look for certain sub-castles contained in  $S_{\Gamma, X, n}$ defined in Example \ref{eg: canonical folner sequence} by induction. However, one cannot proceed in this canonical way using a general F{\o}lner sequence $\CS$ even for $\Z$-systems. For example,  it is impossible using $T_n$ in Example \ref{eg: ill-behaved density} to establish the  almost finiteness in measure for $X\rtimes\Z$ because all multisections induced by $\{(\gamma x, \gamma, x):\gamma\in \{0\}\cup s^n_{k}+F_n, x\in O^n_k\}$ ``collapse'' eventually. In another words, there are too many overlaps for these multisections contained in each $T_n$. We have seen this is also an obstruction in establishing effective connection between densities and invariant measures. These thus suggest that one may need to work with F{\o}lner sets in which multisections should be ``disjoint'' in some sense, and thus leads to the following definition. Given a family $\CA=\{A_i, i\in I\}$ of sets. We say $\CA$ is \emph{at most $n$-colorable} if there are at most $n$ disjoint subfamily $\CA_1,\dots, \CA_n$ of  $\CA$ such that $\CA=\bigcup_{k=1}^n\CA_k$.

\begin{defn}\label{defn: good folner set}
Let $S=\bigcup_{l=1}^m\bigsqcup_{i\in F_l}C^l_{i, i_l}$ be a normal F{\o}lner set in the sense of \ref{defn_normal folner set}.  Define the \textit{height} of $S$ to be $h(S)=\min\{|I_l|, 1\leq l\leq m\}$. Let $M\in \N_+$ and $\epsilon>0$.  We say $S$  is $(M, \epsilon)$\textit{-good}  if there is an injective function $f_l: \{0, \dots, h(S)-1\}\to I_l$ for any $1\leq l\leq m$ such that 
\begin{enumerate}
	\item $f_l(0)=i_l$ for each $1\leq l\leq m$
	\item $\inf_{\mu\in M(\CG)}\mu(\bigcup_{l=1}^mC^l_{f_l(k), f_l(k)})\geq 1-\epsilon$ for any $0\leq k\leq h(S)-1$.
	\item There is a $L\subset \{1,\dots, m\}$ with $\inf_{\mu\in M(\CG)}\mu(\bigsqcup_{l\in L}C^l_{i_l, i_l})>1-\epsilon$ such that $\{C^l_{f_l(k), f_l(k)}: l\in L\}$ is at most $M$-colorable for any $0\leq k\leq h(S)-1$.
\end{enumerate}
\end{defn}

\begin{prop}\label{prop:almost finite in measure admits good Folner sequence}
	Let $\CG$ be a  groupoid that is almost finite in measure and  $\{\epsilon_n: n\in \N\}$ a decreasing sequence with $\lim_{n\to \infty}\epsilon_n=0$. Then $\CG$ admits a F{\o}lner sequence $\CS=\{S_n: n\in \N\}$ such that all $S_n$ are $(1, \epsilon_n)$-good.
\end{prop}
\begin{proof}
It suffices to show that for any compact set $K\subset \CG$ and $\epsilon>0$, there is a $(1, \epsilon)$-good normal $(K, \epsilon)$-F{\o}lner set. Indeed, let such $K$ and $\epsilon>0$ be given. Since $\CG$ is almost finite in measure, there is a compact open elementary groupoid $\CH$ such that
\begin{enumerate}
	\item $|K\CH u\setminus \CH u|<\epsilon|\CH u|$ for any $u\in \HU$ and
	\item $\mu(\HU)>1-\epsilon$ for any $\mu\in M(\CG)$.
\end{enumerate} 
For $\CH$, there is a compact open castle $\CC=\{C^l_{i, j}: i, j\in F_l, l\in I\}$ such that $\CH=\bigcup\CC$. Without loss of generality, we write $F_l=\{1,\dots, k_l\}$ for some $k_l\in \N$.   Define a new index set $J=\{p=(l, i): i\in F_l, l\in I\}$ and for each $p\in J$ define $E_p=F_l$ and a particular index $d_p\in E_p$ by $d_p=i$. Then define the bisection $A^p_{j, d_p}=C^l_{j, i}$ and $A^p_{j, j'}=C^l_{j, j'}$ for $j, j'\in E_p$. Observe that $\CH=\bigsqcup_{p\in J}\bigsqcup_{j\in E_p}A^p_{j, d_p}$. Then using the same method in Proposition \ref{prop:normal folner set}, there are multisections $\CD^q=\{D^q_{m,n}: m,n\in T_q\}$ with particular index $n_q\in T_q$, for each $q\in Q$ where $Q$ is a finite index set, satisfying
\begin{enumerate}
\item $\GU\setminus \HU=\bigsqcup_{q\in Q}D^q_{n_q, n_q}$ and
\item if we write $D=\bigcup_{q\in Q}\bigsqcup_{m\in T_q}D^q_{m, n_q}$ then $\CD u$ is $(K, \epsilon)$-F{\o}lner for any $u\in \GU\setminus \HU$.
\end{enumerate}
Then we define $S=\CH\cup D$, which is a normal $(K, \epsilon)$-F{\o}lner set. Now we verify that $S$ is $(1, \epsilon)$-good. Let $h(S)$ be the height of $S$. Then we first work on $\CH$ part. For each $p=(l, i)\in J$ with $E_p=F_l=\{1, \dots, k_l\}$ define the map $f_p:\{0, \dots h(S)-1\} \to E_p$ by $f_p(j)=i+j \mod k_l$. Then for each $q\in Q$ choose an arbitrary subset $T'_q\subset T_q$ with $|T'_p|=h(S)$ and a bijective map $g_q: \{0, \dots, h(S)-1\} \to T'_q$ with $g_q(0)=n_q$. Observe that \[\bigsqcup_{p\in J}A^p_{f_p(k), f_p(k)}=\bigsqcup_{l\in I}\bigsqcup_{i\in F_l}C^l_{i, i}=\HU\]  for each $0\leq k\leq h(S)$, which entails 
\[\inf_{\mu\in M(\CG)}\mu(\bigsqcup_{p\in J}A^p_{f_p(k), f_p(k)}\cup \bigcup_{p\in J}D^q_{g_q(k), g_q(k)})\geq \inf_{\mu\in M(\CG)}(\HU)>1-\epsilon.\]
Then consider $\CH=\bigsqcup_{p\in J}\bigsqcup_{j\in E_p}A^p_{j, d_p}$, which satisfying \[\inf_{\mu\in M(\CG)} \mu(\bigsqcup_{p\in J}A^p_{d_p, d_p})=\inf_{\mu\in M(\CG)}\mu(\HU)>1-\epsilon.\]
In addition, for each $0\leq k\leq h(S)-1$, the family $\{A^p_{f_p(k), f_p(k)}: p\in J\}$ is exactly the level sets $\{C^l_{i, i}: i\in F_l, l\in I\}$, which is disjoint ant thus has only one color.
\end{proof}

We also need the following concept.

\begin{defn}\label{defn: control length}
	Let $S=\bigcup_{l=1}^m\bigsqcup_{i\in F_l}C^l_{i, i_l}$ be a normal F{\o}lner set in the sense of \ref{defn_normal folner set}. Let $N\in \N_+$ We say $S$ has $N$-\textit{controlled} height if $|I_l|\leq N\cdot h(S)$ for any $1\leq l\leq m$.
\end{defn}

\begin{rmk}\label{rmk: 1-controlled height}
	 We remark that the normal F{\o}lner sets constructed in Proposition \ref{prop: minimal groupoid normal folner} has $1$-controlled height  and thus Proposition \ref{prop: exist folner sequence} implies that there  exists a F{\o}lner sequence consists of normal F{\o}lner sets with the $1$-controlled height for minimal groupoids in which there is a $u\in \GU$ such that $\CG^u_u=\{u\}$. This thus applies to a second countable minimal almost finite groupoid $\CG$, which is always topological principal (see \cite[Remark 6.6]{Matui}). On the other hand, provided a decreasing sequence $\{\epsilon_n: n\in \N\}$ with $\lim_{n\to \infty}\epsilon_n=0$, the same argument of Proposition \ref{prop:almost finite in measure admits good Folner sequence} shows that a minimal almost finite groupoid $\CG$ also admits a F{\o}lner sequence $\CS=\{S_n: n\in \N\}$ in which each $S_n$ is $(1, \epsilon_n)$-good by using the normal F{\o}lner sets in Proposition \ref{prop: almost finite folner sequence}. It is still unknown to the author whether a minimal almost finite groupoids admits a F{\o}lner sequence satisfying both of these two properties. 
\end{rmk}

\begin{prop}\label{prop: density bigger than measure}
Let $\CG$ be a groupoid. Let $M, N\in \N$ and a decreasing sequence $\{\epsilon_n: n\in \N\}$ such that $\lim_{n\to \infty}\epsilon_n=0$. Suppose $\CG$ admits a  F{\o}lner sequence $\CS=\{S_n: n\in \N\}$  in which all $S_n$ are $(M, \epsilon_n)$-good and have $N$-controlled height. Then  $\bar{D}^-_\CS(A)\geq \frac{1}{N}\sup_{\mu\in M(\CG)}\mu(A)$ for any clopen set $A$ in $\GU$.
\end{prop}
\begin{proof}
Let $\epsilon>0$ and $A\subset \GU$ be a  clopen set. Let $n$ big enough such that $\epsilon_n<\epsilon$ and, by definition of $\bar{D}^-_\CS(A)$, there is a $(M, \epsilon_n)$-good normal F{\o}lner set $S_n=\bigcup_{l=1}^m\bigsqcup_{i\in F_l}C^l_{i, i_l}$ with $N$-controlled length such that $\bar{D}^-_\CS(A)+\epsilon>\bar{D}^+_{S_n}(A)$ where $\bar{D}^+_{S_n}(A)=\sup_{u\in \GU}(1/|Su|)\sum_{\gamma\in Su}1_A(r(\gamma))$. By the standard chopping technique,  without loss of generality, we may assume either $C^l_{i, i}\subset A$ or $C^l_{i, i}\cap A=\emptyset$ for any $i\in F_l$ and $1\leq l\leq m$. Note that after this refinement, the F{\o}lner set $S_n$ is still $(M, \epsilon_n)$-good and has $N$-controlled height. First, by definition, one has
\[\bar{D}^+_{S_n}(A)= \max_{l=1}^m\sup_{u\in C^l_{i_l, i_l}}\frac{1}{|S_nu|}\sum_{\gamma\in S_lu}1_A(r(\gamma))=\max_{l=1}^m\frac{|\{C^l_{i, i}\subset A: i\in I_l\}|}{|I_l|}.\]
Then the $(M, \epsilon_n)$-goodness of $S_n$ implies that for each $1\leq l\leq m$ there is an injection $f_l: J_l=\{0, \dots, h(S_n)-1\}\to I_l$ satisfying (1) and (2) in Definition \ref{defn: good folner set}. 

For each $0\leq k\leq h(S_n)-1$ define $A_k=A\cap (\bigcup_{l=1}^m C^l_{f_l(k), f_l(k)})$. This implies $\mu(A)-\epsilon\leq \mu(A_k)$ for any $\mu\in M(\CG)$ because $\inf_{\mu\in M(\CG)}\mu(\bigcup_{l=1}^mC^l_{f_l(k), f_l(k)})>1-\epsilon_n>1-\epsilon$. On the other hand, let $\mu\in M(\CG)$. For each $k$,  our construction implies that 
\[\mu(A_k)\leq\sum\{\mu(C^l_{f_l(k), f_l(k)}): C^l_{f_l(k), f_l(k)}\subset A, 0\leq l\leq m\},\]
which entails
\begin{align*}
h(S_n)(\mu(A)-\epsilon)&\leq \sum_{k=0}^{h(S_n)-1}\sum\{\mu(C^l_{f_l(k), f_l(k)}): C^l_{f_l(k), f_l(k)}\subset A, 0\leq l\leq m\}\\
&=\sum_{l=1}^m|\{C^l_{i, i}\subset A: i\in f_l(J_l)\}|\mu(C^l_{i_l, i_l}).
\end{align*}
Then using $|I_l|\leq N\cdot h(S_n)$ for each $1\leq l\leq m$ one actually has 
\begin{align*}
	\mu(A)-\epsilon&\leq \max_{l=1}^m\frac{|\{C^l_{i, i}\subset A: i\in I_l\}|}{h(S_n)}\cdot\sum_{l=1}^m\mu(C^l_{i_l, i_l})
	=\max_{l=1}^m\frac{|\{C^l_{i, i}\subset A: i\in I_l\}|}{h(S_n)}\\
	&\leq \max_{l=1}^m\frac{N\cdot|\{C^l_{i, i}\subset A: i\in I_l\}|}{|I_l|}\leq N\cdot \bar{D}^+_{S_n}(A)\leq N\cdot (\bar{D}^-_\CS(A)+\epsilon)
\end{align*}
Now, one actually has $\sup_{\mu\in M(\CG)}\mu(A)\leq N\cdot (\bar{D}^-_\CS(A)+\epsilon)+\epsilon$. Now let $\epsilon\to 0$, we have $\sup_{\mu\in M(\CG)}\mu(A)\leq N\cdot \bar{D}^-_\CS(A)$ as desired.
\end{proof}

Combining Lemma \ref{lem: density less measure} and Proposition \ref{prop: density bigger than measure}, we have the following result.
\begin{prop}\label{prop: measure equal density}
Let $\CG$ be a  groupoid. Let $M, N\in \N$ and a decreasing sequence $\{\epsilon_n: n\in \N\}$ such that $\lim_{n\to \infty}\epsilon_n=0$. Suppose there is a  F{\o}lner sequence $\CS=\{S_n: n\in \N\}$  in which all $S_n$ are $(M, \epsilon_n)$-good and have $N$-controlled height. Then for any clopen set $A\subset \GU$, one has
\[\frac{1}{N}\sup_{\mu\in M(\CG)}\mu(A)\leq \bar{D}^-(A)\leq \bar{D}^+(A)\leq \sup_{\mu\in M(\CG)}\mu(A).\] In particular, if $N=1$, one has 
$\sup_{\mu\in M(\CG)}\mu(A)=\bar{D}^-(A)= \bar{D}^+(A).$
\end{prop}

\begin{rmk}\label{rmk: right def of density}
Let $M\in \N$ and a decreasing sequence $\{\epsilon_n: n\in \N\}$ such that $\lim_{n\to \infty}\epsilon_n=0$. Let $\CG$ be a $\sigma$-compact groupoid, equipped with a F{\o}lner sequence $\CS=\{S_n: n\in \N\}$ in which all $S_n$ are $(M, \epsilon_n)$-good and have $1$-controlled height. Then \ref{prop: measure equal density} allows us to define upper density and lower density with respect to $\CS$ for any clopen set $A$ by
\[\bar{D}_\CS(A)=\lim_{n\to \infty}\sup_{u\in \GU}\frac{1}{|S_nu|}\sum_{\gamma\in S_nu}1_A(r(\gamma))\]
and
\[\barbelow{D}_\CS(A)=\lim_{n\to \infty}\inf_{u\in \GU}\frac{1}{|S_nu|}\sum_{\gamma\in S_nu}1_A(r(\gamma)).\]
In addition the value of two densities for a given clopen set $A$ does not depend on the choice of F{\o}lner sequences $\{S_n:n\in \N\}$ in which  all $S_n$ are  $(M, \epsilon_n)$-good for some $M\in \N$ and a decreasing sequence $\epsilon_n\to 0$ and have $1$-controlled height.
\end{rmk}

\section{Tilings for ample \'{e}tale groupoids}
In this section, we establish our first main result, i.e., Theorem A. We mainly follow the strategy introduced in \cite{C-J-K-M-S-T} and \cite{D-G} to find proper subcastles in ``good'' F{\o}lner sequences by induction. We begin with the following definition.  

\begin{defn}\label{defn: folner set unit space}
Let $K$ be a compact set in $\CG$ and $\delta>0$. We say a set $A$ in $\GU$ is $(K, \delta)^*$-invariant if for any compact set $F$ in $\CG$ and any $\epsilon>0$ there is a normal $(F, \epsilon)$-F{\o}lner set $S$ such that 
\[\sup_{u\in \GU}\frac{|r(KA\setminus A)\cap r(Su)|}{|A\cap r(Su)|}<\delta.\]
\end{defn}

\begin{prop}\label{prop: folner set in unit space}
Let $\CG$ be a  groupoid that admits a F{\o}lner sequence $\CS=\{S_n: n\in \N\}$  and $K, W$ compact sets in $\CG$. Let $\epsilon, \delta,\delta_0>0$ with $\delta>\delta_0$ and $C\subset s(W)$. Now for any $c\in C$, Let $F_c\subset Wc$ be a $(K, \delta_0(1-\epsilon))$-F{\o}lner set such that $r|_{F_c}$ is injective. Suppose the  collection $\{r(F_c): c\in C\}$ is $\epsilon$-disjoint and the set $A=\bigcup_{c\in C}r(F_c)$ satisfies $\barbelow{D}^-_\CS(A)>0$.   Then $A$ is  $(K, \delta)^*$-invariant.
\end{prop}
\begin{proof}
Set $A=\bigcup_{c\in C}r(F_c)$ and $T=WW^{-1}K^{-1}\cup \GU$. Let $L\subset \CG$ be a compact set and $\epsilon_0>0$.  Then there is a normal F{\o}lner set $V=S_n\in \CS$ for a large enough $n$ such that 
\begin{enumerate}
\item $V$ is a normal $(L, \epsilon)$-F{\o}lner set.
\item $Vu$ is additionally $(T, \frac{\barbelow{D}^-_\CS(A)}{2\|T^{-1}\|}(\delta-\delta_0))$-F{\o}lner for any $u\in \GU$.
\item $\inf_{u\in \GU}\frac{|A\cap r(Vu)|}{|Vu|}>\frac{\barbelow{D}^-_\CS(A)}{2}$.
\end{enumerate}
Let $u\in \GU$ and define $B=\{b\in r(Vu): r(Tb)\not\subset r(Vu)\}$. Note that for any $b\in B$ there is a $\gamma\in T$ such that $r(\gamma b)\notin r(Vu)$ and thus one has 
\[r(\gamma b)\in r(TVu)\setminus r(Vu)\subset r(TVu\setminus Vu),\]
which yields $b\in r(T^{-1} r(TVu\setminus Vu))$. Thus, one has $B\subset r(T^{-1}(TVu\setminus Vu))$. Then using Remark \ref{5.2} and properties (1), (2) for $V$ above, one has 
\begin{align*}
|B|&\leq |T^{-1}(TVu\setminus Vu)|\leq \|T^{-1}\|\cdot|TVu\setminus Vu|\\
&= \|T^{-1}\|\cdot \frac{|TVu\setminus Vu|}{|Vu|}\cdot \frac{|r(Vu)|}{|A\cap r(Vu)|}\cdot|A\cap r(Vu)|\\
&\leq  \|T^{-1}\|\cdot\frac{\barbelow{D}^-_\CS(A)}{2\|T^{-1}\|}(\delta-\delta_0)\cdot \frac{2}{\barbelow{D}^-_\CS(A)}|A\cap r(Vu)|=(\delta-\delta_0)|A\cap r(Vu)|.
\end{align*}
Now, define $C'=\{c\in C: r(F_c)\subset r(Vu)\}$. Then the $\epsilon$-disjointness of $\{r(F_c)\}$, the fact that $r|_{F_c}$ are all injective, and $\bigcup_{c\in C'}r(F_c)\subset A\cap r(Vu)$ imply that \[(1-\epsilon)\sum_{c\in C'}|F_c|=(1-\epsilon)\sum_{c\in C'}|r(F_c)|\leq |A\cap r(Vu)|.\] 
Now we claim if $c\in C\setminus C'$ then $r(KF_c)\cap r(Vu)\subset B$. Indeed, if $r(KF_c)\cap r(Vu)\neq \emptyset$, let $b\in r(KF_c)\cap r(Vu)$ where $c\in C\setminus C'$. One can write $b=r(\gamma x)$ where $x\in F_c$ and $\gamma\in K$. Since $c\notin C'$, there is a $y\in F_c$ such that $r(y)\notin r(Vu)$. Then using $s(y)=s(x)=c$, one actually has $r(y)=r(yx^{-1}\gamma^{-1}b)\in r(Tb)$. This implies that $r(Tb)\not\subset r(Vu)$. Therefore one has $r(KF_c)\cap r(Vu)\subset B$ by the definition of $B$.

Now, let $v\in (r(KA)\setminus A)\cap r(Vu)$. Write $v=r(\gamma x)\notin A$ such that $\gamma\in K$ and $x\in F_c$  for some $c\in C$. If $c\in C'$ then $v\in r(KF_c)\setminus r(F_c)$ because $v\notin A$. On the other hand, if $c\in C\setminus C'$ then $v\in r(KF_c)\cap r(Vu)\subset B$. This implies 
\[(r(KA)\setminus A)\cap r(Vu)\subset B\cup \bigcup_{c\in C'}(r(KF_c)\setminus r(F_c))\]
in which the cardinality of the latter set satisfies
\begin{align*}
|\bigcup_{c\in C'}(r(KF_c)\setminus r(F_c))|&\leq  \sum_{c\in C'}|r(KF_c)\setminus r(F_c)|\leq \sum_{c\in C'}|r(KF_c\setminus F_c)|\\
&\leq \sum_{c\in C'}\delta_0(1-\epsilon)|F_c|\leq \delta_0|A\cap r(Vu)|.
\end{align*}
Therefore, one has
 \begin{align*}
 |(r(KA)\setminus A)\cap r(Vu)|&\leq |B|+|\bigcup_{c\in C'}(r(KF_c)\setminus r(F_c))|\\
 &\leq (\delta-\delta_0)|A\cap r(Vu)|+\delta_0|A\cap r(Vu)|\leq\delta|A\cap r(Vu)|
 \end{align*}
 Now since $u$ is arbitrary, one has $A$ is $(K, \delta)^*$-invariant.
\end{proof}

\begin{prop}\label{prop: estimate of density}
	Let $\CG$ be a groupoid that admits a F{\o}lner sequence $\CS=\{S_n: n\in\N\}$ in which all $S_n$ are $(M, \epsilon_n)$-good and have $1$-controlled height for some $M\in \N$ and a decreasing sequence $\{\epsilon_n\}$ satisfying $\lim_{n\to \infty}\epsilon_n=0$. Let $0<\epsilon, \delta, \eta<1$ satisfying  $\epsilon(1+\delta)<1$.  Let $M\in \N_+$ and $T=\bigcup_{l=1}^m\bigsqcup_{i\in F_l}C^l_{i, i_l}$ be a $(M, \eta)$-good $1$-controlled normal F{\o}lner set. Suppose $A\subset \GU$ is $(T^{-1}, \delta)^*$-invariant and $B\supset A$ is a clopen set satisfying $|B\cap r(Tu)|\geq \epsilon|Tu|$ for any $u\in \GU$. Then $\barbelow{D}_\CS(B)\geq (1-(\epsilon/M))(1+\delta))\barbelow{D}_\CS(A)+(1-\eta)\epsilon/M$.
\end{prop}
\begin{proof}
First Remark \ref{rmk: right def of density} allows to define densities $\bar{D}_\CS(\cdot)$ and $\barbelow{D}_\CS(\cdot)$ at least for clopen sets. Let $\theta>0$. Using the terminology of Definition \ref{defn: good folner set}, for $T$ we denote by $P=\bigsqcup_{l\in L}C^l_{i_l. i_l}$ with $\inf_{\mu\in M(\CG)}\mu(P)\geq 1-\eta$ such that the family $\{C^l_{k, k}: l\in L\}$ is at most $M$-colorable for each $0\leq k\leq h(T)-1$. On the other hand, since $T$ has $1$-controlled height, one has $|I_l|=h(T)$ is a constant for any $1\leq l\leq m$. Then to simplify our notation, we may identify each $I_l$ by $\{0, \dots, h(T)-1\}$ in the way that $f_l$ defined in Definition \ref{defn: good folner set} is the identity.

Now Proposition \ref{prop: measure equal density}, Remark \ref{rmk: right def of density} and Definition \ref{defn: folner set unit space} imply that there is an $N\in \N$ such that for any $n>N$ the normal $(T, \theta)$-F{\o}lner set $F=S_n$ satisfies 
	\begin{enumerate}
	\item $\inf_{u\in \GU}\frac{|A\cap r(Fu)|}{|Fu|}>\barbelow{D}_\CS(A)-\theta$.
	\item $\sup_{u\in \GU}\frac{|r(T^{-1}A)\cap r(Fu)|}{|A\cap r(Fu)|}<1+\delta$.
	\item $\sup_{u\in \GU}\frac{|(\GU\setminus P)\cap r(Fu)|}{|Fu|}<\eta+\theta$.
	\end{enumerate}
Now fix a $u_0\in \GU$ and set $\alpha=\frac{|A\cap r(Fu_0)|}{|Fu_0|}\geq \barbelow{D}_\CS(A)-\theta$. Then we define $F'=\{x\in Fu_0: A\cap r(Tx)=\emptyset\}$. We claim $F'=Fu_0\setminus \{x\in Fu_0: r(x)\in r(T^{-1}A)\}$. Indeed, let $x\in Fu_0$. If $u\in A\cap r(Tx)$ then one can write $u=r(\eta)\in A$ for some $\eta\in T$ with $s(\eta)=r(x)$. Then one has $r(x)=s(\eta)=r(\eta^{-1})\in r(T^{-1}A)$ since $s(\eta^{-1})=r(\eta)=u\in A$. This implies that if $A\cap r(Tx)\neq \emptyset$ then $r(x)\in r(T^{-1}A)$. On the other hand, if $r(x)\in r(T^{-1}A)$, one writes $r(x)=r(\gamma)$ for some $\gamma\in T^{-1}$ with $s(\gamma)\in A$. Then $r(\gamma^{-1}x)=r(\gamma^{-1})=s(\gamma)\in A$. Because $\gamma^{-1}\in T$, one actually has $A\cap r(Tx)\neq \emptyset$. We have verified that $r(x)\in r(T^{-1}A)$ is equivalent to $A\cap r(Tx)\neq \emptyset$ and thus $F'=Fu_0\setminus \{x\in Fu_0: r(x)\in r(T^{-1}A)\}$.

Note that $|\{x\in Fu_0: r(x)\in r(T^{-1}A)\}|=|r(Fu_0)\cap r(T^{-1}A)|$ and thus one has 
\begin{align*}
\frac{|F'|}{|Fu_0|}&=\frac{|Fu_0|-|r(Fu_0)\cap r(T^{-1}A)|}{|Fu_0|}=1-\frac{|r(Fu_0)\cap r(T^{-1}A)|}{|A\cap r(Fu_0)|}\cdot\frac{|A\cap r(Fu_0)|}{|r(Fu_0)|}\\
&\geq 1-(1+\delta)\alpha.
\end{align*}
Now define $F''=\{x\in F': r(x)\in P\}$. Then $(3)$ above implies that \[|F''|\geq |F'|-(\eta+\theta)|Fu_0|\geq ((1-(1+\delta)\alpha)-(\eta+\theta))|Fu_0|.\]
Then since $A\cap r(Tx)=\emptyset$ for any $x\in F''\subset F'$ and $|B\cap r(Tu)|\geq\epsilon|Tu|$ holds for any $u\in \GU$ by assumption,  for any $x\in F''$, one has 
\[\sum_{\eta\in Tx}1_{B\setminus A}(r(\eta))=|(B\setminus A)\cap r(Tx)|\geq \epsilon|Tx|.\]
Then one has 
\[\sum_{x\in F''}\sum_{\eta\in Tx}1_{B\setminus A}(r(\eta))\geq \epsilon|Tx||F''|=\epsilon|F''|h(T).\]
Then for any $x\in F''$, there is an $l\in L$ such that $r(x)\in C^l_{i_l, i_l}$. Recall we have identified each $I_l$ with $\{0, \dots, h(T)-1\}$ above, one has $Tx=\bigsqcup_{k=0}^{h(T)-1}C^l_{k, 0}x$. Denote by $\{\eta^k_x\}=C^l_{k, 0}r(x)$. Then one has
\[\sum_{k=0}^{h(T)-1}\sum_{x\in F''}1_{B\setminus A}(r(\eta^k_x))\geq\epsilon|F''|h(T).\]
This implies that there is a $0\leq k\leq h(T)-1$ such that $\sum_{x\in F''}1_{B\setminus A}(r(\eta^{k}_x))\geq\epsilon|F''|$.
Since the family $\{C^l_{k, k}: l\in L\}$ is at most $M$-colorable, one has
\[|(B\setminus A)\cap \{r(\eta^{k}_x): x\in F''\}|\geq \frac{\epsilon}{M}|F''|.\]
Note that all $\eta^{k}_x\in T$. Therefore, one has 
\begin{align*}
\frac{|B\cap r(TFu_0)|}{|Fu_0|}&\geq \frac{|A\cap r(Fu_0)|}{|Fu_0|}+\frac{|(B\setminus A)\cap \{r(\eta^{k}_x): x\in F''\}|}{|F''|}\cdot \frac{|F''|}{|Fu_0|}\\
&\geq\alpha+ (\epsilon/M)(1-(1+\delta)\alpha-(\eta+\theta))\\
&\geq(1-(\epsilon/M)(1+\delta))(\barbelow{D}_\CS(A)-\theta)+(\epsilon/M)(1-\eta+\theta).\end{align*}
Finally, since $\GU\subset T$ and $F$ is a normal $(T, \theta)$-F{\o}lner set, observe that 
\begin{align*}
|B\cap r(TFu_0)|-|B\cap r(Fu_0)|&=|(B\cap r(TFu_0))\setminus(B\cap r(Fu_0))|\\
&\leq |r(TFu_0)\setminus r(Fu_0)|\leq |r(TFu_0\setminus Fu_0)|\leq \theta|Fu_0|.
\end{align*}
Therefore, one has
\[\frac{|B\cap r(Fu_0)|}{|Fu_0|}\geq (1-(\epsilon/M)(1+\delta))(\barbelow{D}_\CS(A)-\theta)+(\epsilon/M)(1-\eta+\theta)-\theta.\]
Now since $u_0$ is arbitrary, one actually has 
\[\barbelow{D}_\CS(B)\geq (1-(\epsilon/M)(1+\delta))(\barbelow{D}_\CS(A)-\theta)+(\epsilon/M)(1-\eta+\theta)-\theta.\]
Now let $\theta\to 0$ one has $\barbelow{D}_\CS(B)\geq (1-(\epsilon/M)(1+\delta))\barbelow{D}_\CS(A)+(\epsilon/M)(1-\eta).$
\end{proof}

Now we introduce several operations on multisections and castles.

\begin{defn}\label{defn: sub-castles}
Let $\CC=\{C_{i, j}: i, j\in F\}$ and $\CD=\{D_{i, j}: i, j\in E\}$ be multisections. We say $\CD$ is \textit{contained in} $\CC$ if $E\subset F$ and $D_{i, j}\subset C_{i, j}$ for any $i, j\in E$. Let $\CT=\{\CD_1,\dots, \CD_n\}$ be a castle. We say $\CT$ is contained in a multisection $\CC$ if all multisections $\CD_i$ inside $\CT$ are contained in $\CC$.
\end{defn}

\begin{defn}\label{defn: restriction of towers}
	Let $\CC=\{C_{i, j}: i, j\in F\}$ be a multisection. Let $V\subset C_{i_0, i_0}$ for some $i_0\in F$ and $T\subset F$, we write $(\CC|_{V}, T)$ for the sub-mutisection $\{B_{i, j}, i, j\in T\}$ in which each $B_{i, j}= C_{i, i_0}V(C_{j, i_0})^{-1}$. We remark that $i_0$ may or may not in $T$. 
\end{defn}

Let $\CC$ be a castle Recall our notation in Remark \ref{rmk: simplify notation}, for simplicity, we write $\CC u$=$\CH_\CC u$ for any $u\in \HU_\CC=\bigcup\CUU$ where $\CH_\CC=\bigcup \CC$.

\begin{lem}\label{lem: construct castle first step}
Let  $\CG$ be a  groupoid. Suppose $0<\epsilon<1/2$ and $S$ is a compact open set in $\CG$ such that there is a family $\{C^l_{i, j}: i, j\in F_l, 1\leq l\leq m\}$ of compact open bisections with an $i_l\in F_l$ for each $1\leq l\leq m$ such that 
\begin{enumerate}
	\item $\CC_l=\{C^l_{i, j}: i, j\in F_l\}$ is a compact open multisection for each $1\leq l\leq m$;
	\item  $S=\bigcup_{l=1}^m\bigsqcup_{i\in F_l}C^l_{i, i_l}$ and $\GU=\bigsqcup_{l=1}^mC^l_{i_l, i_l}.$
\end{enumerate} 
Let $Y$ be a clopen set in $\GU$. Then there is a compact open castle $\CT=\{\CQ_1,\dots, \CQ_n\}$  such that 
\begin{enumerate}
	\item each $\CQ_l$ itself is a castle contained in $\CC_l$ for each $1\leq l\leq m$ in the sense of Definition \ref{defn: sub-castles} and $|\CQ_l u|\geq (1-\epsilon)|F_l|$ for any $u\in \bigcup\QU_l$.
	\item 
	The set $A=\bigsqcup_{l=1}^m\bigcup\QU_l$ satisfies $Y\cap A=\emptyset$ 
	 and $|A\cap r(Su)|\geq \epsilon|Su|$ for any $u\in \GU$.
\end{enumerate}
\end{lem}
\begin{proof}
	Let $S=\bigcup_{l=1}^m\bigsqcup_{i\in F_l}C^l_{i, i_l}$ be the compact open set above. Denoted by $V_l$ all the $C^l_{i_l, i_l}$ for simplicity. Then $\{V_l: l=1\dots, m\}$ forms a clopen partition of $\GU$. Define $A_0=Y$ and $\CQ_0=\{A_0\}$. We then recursively define $A_1,\dots, A_m$ as well as castles
	\[\CQ_k=\{\CB_{T, k}=\{B^{T, k}_{i, j}: i, j\in T\}: T\subset F_k, |T|\geq (1-\epsilon)|F_k|\}\]
	so that $A_{k+1}=A_k\sqcup \bigcup\QU_k$ for $1\leq k\leq m-1$.
	Now, suppose we have constructed sets $A_0,\dots, A_{l-1}$ and corresponding $\CQ_0,\dots, \CQ_{l-1}$ with the desired property. Now for $V_l=C^l_{i_l, i_l}$ and $T\subset F_l$ with $|T|\geq (1-\epsilon)|F_l|$, define
	\[V_{T, l}=V_l\cap \bigcap_{i\in F_l\setminus T}r((C^l_{i, i_l})^{-1}A_{l-1})\cap \bigcap_{i\in T}(\GU\setminus r((C^l_{i, i_l})^{-1}A_{l-1})).\]
	Note that $V_{T, l}$ may be empty and we define $\CB_{T, l}=(\CC_l|_{V_{T, l}}, T)=\{B^{T, l}_{i, j}: i, j\in T\}$ in which  $B^{T, l}_{i, j}=C^l_{i, i_l}V_{T, l}(C^l_{j, i_l})^{-1}$ for any $i, j\in T$. We warn that $i_l$ may not in $T$.
	
	First we claim that $\{V_{T, l}: T\subset F_l, |T|\geq (1-\epsilon)|F_l|\}$ is disjoint. Indeed, for such $T_1, T_2\subset F_l$, if $T_1\neq T_2$, without loss of generality, one may assume there is an $i\in T_1\setminus T_2$. Now if $u\in V_{T_1, l}\cap V_{T_2, l}\neq \emptyset$, then by definition, one actually has $u\notin r((C^l_{i, i_l})^{-1}A_{l-1})$. On the other hand, $i\notin T_2$ implies that $u\in r((C^l_{i, i_l})^{-1}A_{l-1})$, which is a contradiction. This establishes our claim and thus
	\[\CQ_l=\{\CB_{T, l}=\{B^{T, l}_{i, j}: i, j\in T\}: T\subset F_l, |T|\geq (1-\epsilon)|F_l|\}\] 
	is a castle because $\CC_l$ is a multisection.
	
	Then we claim $A_{l-1}$ is disjoint from $\bigcup\QU_l$. Suppose not, let $u\in B^{T, l}_{i, i}\cap A_{l-1}$ for some $i\in T\subset  F_l$ such that $|T|\geq (1-\epsilon)|F_l|$. Note that $B^{T, l}_{i, i}=r(C^l_{i, i_l}V_{T, l})$, which implies that $\{u\}=r(C^l_{i, i_l}v)$ for some $v\in V_{T, l}\subset \GU\setminus r((C^l_{i, i_l})^{-1}A_{l-1})$ since $i\in T$. However, note that $\{v\}=r((C^l_{i, i_l})^{-1}u)\subset r((C^l_{i, i_l})^{-1}A_{l-1})$. This is a contradiction and thus establishes the claim that $A_{l-1}\cap \QU_l=\emptyset$.
	
	Finally, we define $A_l=A_{l-1}\sqcup \bigcup\QU_l$. This finishes the induction process. Now define $A=\bigsqcup_{l=1}^m\bigcup\QU_l$, which satisfies $Y\cap A=\emptyset$ by our construction. 
	
 It is left to show $|A\cap r(Su)|\geq \epsilon|Su|$ for any $u\in \GU$. Let $u\in \GU$. Then there is a $V_l$ such that $u\in V_l$ since $\{V_l: l=1,\dots, m\}$ form a partition of $\GU$. Suppose $u\in V_{T, l}$ for some $T\subset F_l$ with $|T|\geq (1-\epsilon)|F_l|$. Then because $V_{T, l}\subset V_l=C^l_{i_l, i_l}$,  one has \[|\CB_{T, l}u|=|T|\geq (1-\epsilon)|F_l\geq \epsilon|F_l|=\epsilon|Su|.\]
On the other hand, if $u\notin V_{T, l}$ for any $T\subset F_l$ with $|T|\geq (1-\epsilon)|F_l|$ then $|\{i\in F_l: r(C^l_{i, i_l}u)\subset A_{l-1}\}|\geq \epsilon|F_l|$ must hold. Otherwise, if we write $T=\{i\in F_l: r(C^l_{i, i_l}u)\not\subset A_{l-1}\}$ then $T$ has to satisfy $|T|\geq (1-\epsilon)|F_l|$. We remark that each $r(C^l_{i, i_l}u)$ is a singleton. Now it is not hard to see $u\in V_{T, l}$, which is a contradiction to our assumption on $u$ in the first place. Thus we have 
	\[|A\cap r(Su)|\geq |\{i\in F_l: r(C^l_{i, i_l}u)\subset A_{l-1}\}|\geq \epsilon|F_l|=\epsilon|Su|.\]
\end{proof}

\begin{rmk}\label{5.8}
Using the same notations of Lemma \ref{lem: construct castle first step}, let $S=\bigcup_{l=1}^m\bigsqcup_{i\in F_l}C^l_{i, i_l}$ with $\GU=\bigsqcup_{l=1}^mC^l_{i_l, i_l}$. The Lemma  actually allows to construct a castle $\CT=\{\CQ_1,\dots, \CQ_m\}$ ``inside'' $S$ in the sense that each sub-castle $\CQ_l$ 
is contained in $\CC_l=\{C^l_{i, j}: i, j\in F_l\}$ in the sense of Definition \ref{defn: sub-castles}. We will apply this to obtain a quasi-tiling of groupoids in the following theorem. 
\end{rmk}

\begin{thm}\label{thm: almost finite in measre}
Let $\CG$ be a  groupoid that admits a F{\o}lner sequence $\CS=\{S_n: n\in\N\}$ in which all $S_n$ are $(M, \epsilon_n)$-good and have $1$-controlled height for some $M\in \N$ and a decreasing sequence $\{\epsilon_n: n\in \N\}$ satisfying $\lim_{n\to \infty}\epsilon_n=0$.  Then $\CG$ is almost finite in measure.
\end{thm}
\begin{proof}
	Let $K$ be a compact set in $\CG$ and $0<\epsilon<1/2$. Now choose a $0<\eta<\epsilon$. Then there is an $N\in \N$ such that for all $n>N$ one has $\epsilon_n<\eta$ so that the F{\o}lner set $S_n$ is in fact $(M, \eta)$-good. Thus, without loss of any generality, we may assume all $S_n\in \CS$ is $(M, \eta)$-good. Choose a $0<\beta<(\epsilon-\eta)/(1-\epsilon)$ and an $n\in \N_+$ such that $(1-\epsilon/M)^n<1-\frac{(1-\epsilon)(1+\beta)}{1-\eta}$ and $\epsilon(1+\beta)<1$. Our construction implies that $(1+\beta)^{-1}(1-(1-\frac{(1+\beta)\epsilon}{M})^n)(1-\eta)>1-\epsilon$.  Choose another number $0<\beta_0<\beta$. Then by induction, one can find normal F{\o}lner sets $S_1, \dots, S_n$ such that for each $1\leq p\leq n$ one has
	\begin{enumerate}[label=(\roman*)]
		\item $\CC^p_l=\{C^l_{i, j}: i, j\in F^p_l\}$ is a compact open multisection for each $1\leq l\leq m_p$;
		\item $S_p=\bigcup_{l=1}^{m_p}\bigsqcup_{i\in F_l^p}C^{l, p}_{i, i^p_l}$ 
		\item $\GU=\bigsqcup_{l=1}^{m_p}C^{l, p}_{i^p_l, i^p_l}$; 
		\item $S_pu$ is $(K, \epsilon-\beta_0\epsilon)$-F{\o}lner for any $u\in \GU$; 
		\item $S_pu$ is $(S^{-1}_q, \beta_0(1-\epsilon)-\beta_0\epsilon)$-F{\o}lner for any $q<p$ and $u\in \GU$; and
		 \item each $S_p$ is $(M, \eta)$-good and has $1$-controlled height.
	\end{enumerate}
Now, by a recursive procedure, we will construct from $n$ to $1$ compact open castles $\CT_n,\dots,\CT_1$ such that
\begin{enumerate}
	\item $\bigcup\TU_n,\dots, \bigcup\TU_1$ are disjoint sets.
	\item each $\CT_p$ is ``inside'' $S_p$ in the sense of Remark \ref{5.8} and $|\CT_p u|\geq (1-\beta\epsilon)|F^p_l|$ whenever $u\in (\bigcup\TU_p)\cap (\bigcup(\CC^p_l)^{(0)})$ with $1\leq l\leq m_p$.
	\item $|(\bigcup\CT^{(0)}_p)\cap r(S_pu)|\geq \epsilon|S_pu|$ for any $u\in \GU$ and $1\leq p\leq n$.
	\item $\barbelow{D}_\CS(\bigcup_{p=k}^n\bigcup\TU_p)\geq (1+\beta)^{-1}(1-(1-(\epsilon/M)(1+\beta))^{n+1-k})(1-\eta)$ for each $1\leq k\leq n$.
\end{enumerate}
	First, plug $Y=\emptyset$ and $S=S_n$ into Lemma \ref{lem: construct castle first step} to obtain a castle $\CT_n$ satisfying (1)-(3) above. Define $A_n=\bigcup\TU_n$. Then Proposition \ref{prop: estimate of density} implies that $\barbelow{D}^-_\CS(A_n)\geq (\epsilon/M)(1-\eta)$, which satisfies (4) as well.
	
	Now, let $1\leq k\leq n$ and suppose we have constructed $\CT_{n}, \dots, \CT_{k+1}$ satisfying (1)-(4) above. We define $A_p=\bigcup\TU_p$ for all $k+1\leq p\leq n$. Now plug  $Y=\bigsqcup_{p=k+1}^nA_p$ and $S=S_k$ in Lemma \ref{lem: construct castle first step} to obtain a castle 
	\[\CT_k=\{\CB^k_{T, l}=\{B^{k, T, l}_{i, j}: i, j\in T \}: T\subset F^k_l, |T|\geq (1-\beta_0\epsilon)|F^k_l|, l=1,\dots, m_k\}\]
	such that, if write $A_k$ for $\bigcup\TU_k$, then one has $A_k\cap (\bigsqcup_{p=k+1}^nA_p)=\emptyset$ and $|A_k\cap r(S_ku)|\geq \epsilon|S_ku|$ for any $u\in \GU$.
	
	Now write $B=\bigsqcup_{p=k}^nA_p$ and $A=\bigsqcup_{p=k+1}^nA_p$.
	Let $p>k$. For any $u\in \bigcup\TU_p$, (2) above implies that there is an $l\leq m_p$ and an $i\in F^p_l$  such that $u\in C^{l, p}_{i, i}$ and $|\CT_pu|\geq (1-\beta_0\epsilon)|F^p_l|$. Let $\gamma\in C^{l, p}_{i^p_l, i}$ with $s(\gamma)=u$. Then $r(\gamma)\in C^{l, p}_{i^p_l, i^p_l}$. Observe that $\CT_pu\gamma^{-1}\subset S_pr(\gamma)$ and satisfies \[|\CT_pu\gamma^{-1}|=|\CT_pu|\geq (1-\beta_0\epsilon)|F^p_l|=(1-\beta_0\epsilon)|S_pr(\gamma)|.\]
	Then Lemma \ref{lem: shrink folner set} implies that $\CT_pu$ is $(S^{-1}_k, \beta_0(1-\epsilon))$-F{\o}lner as well as $(K, \epsilon)$-F{\o}lner  by the setting (iv) and (v) for all $S_n,\dots, S_1$.
	Now  for each $p>k$, choose a transversal $D_p$ of each castle $\CT_p$ in the sense of Definition \ref{defn: transversal}. Then $A=\bigsqcup_{p=k+1}^n\bigcup\TU_p=\bigsqcup_{p=k+1}^n\bigsqcup\{r(\CT_pu): u\in D_p\}$ in which $\{r(\CT_pu): u\in D_p, k+1\leq p\leq n\}$ is a disjoint family. In addition, one has $\barbelow{D}_\CS(A)\geq \barbelow{D}_\CS(A_n)\geq (\epsilon/M)(1-\eta)>0$. Then
	 Proposition \ref{prop: folner set in unit space} implies that $A$ is $(S_k^{-1}, \beta)^*$-invariant. Finally, because $|B\cap r(S_ku)|\geq |A_k\cap r(S_ku)|\geq \epsilon|S_ku|$ for any $u\in \GU$. Then Proposition \ref{prop: estimate of density} implies
	 \begin{align*}
	 	\barbelow{D}_\CS(B)&\geq (1-(\epsilon/M)(1+\beta))\barbelow{D}_\CS(A)+(\epsilon/M)(1-\eta)\\
	 	&\geq (1-\frac{\epsilon(1+\beta)}{M})(1+\beta)^{-1}(1-(1-\frac{\epsilon(1+\beta)}{M})^{n+1-(k+1)})(1-\eta)+\frac{\epsilon(1-\eta)}{M}\\
	 	&=(1+\beta)^{-1}(1-(1-(\epsilon/M)(1+\beta))^{n+1-k})(1-\eta)
	 \end{align*}
as desired. Thus we have finished the construction. 

Now we write $\CT=\{\CT_1,\dots, \CT_n\}$, which is a castle and denote by $\CH_\CT$ the subgroupoid generated by $\CT$ such that $\HU_\CT=\bigsqcup_{p=1}^n\bigcup\TU_p$. Then, for any $u\in \HU_\CT$, there is a $1\leq p\leq n$ such that $\CH_\CT u=\CT_pu$ and thus our construction above implies that 
\[|K\CH_\CT u\setminus\CH_\CT u|=|K\CT_pu\setminus \CT_pu|\leq \epsilon|\CT_pu|=\epsilon|\CH_\CT u|.\]
In addition, one has
\[\barbelow{D}_\CS(\HU_\CT)=\barbelow{D}_\CS(\bigsqcup_{p=1}^n\bigcup\TU_p)\geq (1+\beta)^{-1}(1-(1-(1+\beta)(\epsilon/M))^n)(1-\eta)>1-\epsilon,\]
which implies that $\inf_{\mu\in M(\CG)}(\HU_\CT)>1-\epsilon$ by Remark \ref{rmk: right def of density}.
\end{proof}
Finally, we provide an example on groupoids generated by partial dynamical systems.

\begin{eg}\label{eg: partial dynamical system}
	Let $\Gamma$ be a finitely generated infinite amenable group with generators $g_1,\dots, g_n$. Let $X$ be the Cantor set.  For each $1\leq i\leq n$, assume $g_i$ acts on $X$ in a way that  the $g_ix$ is defined for all but finitely many points in $X$, which means
	 \[g_i: X\setminus \{x^i_1, \dots, x^i_{n_i}\}\to X\setminus\{y^i_1, \dots, y^i_{n_i}\}\]
	 is a homeomorphism. Then these naturally induce a partial dynamical system of $\Gamma$ on $X$, denoted by $\alpha$. Then we consider the groupoid $\CG=X\rtimes_\alpha \Gamma$. Suppose $\alpha$ is free and  for each unit $x\in X$, the source orbit $\CG x$ is infinite.  We claim $\CG$ is almost finite in measure.
	 
	 Indeed, first observe that for any finite set $F\subset \Gamma$, there are all but only finitely many points $x\in X$ such that $gx$ is well-defined for any $g\in F$. Now let $\{F_n:n\in \N\}$ be a F{\o}lner sequence of $\Gamma$ with $e_\Gamma\in F_n$.  Let $n\in \N^+$, we enumerate all points $x$ such that $F_nx$ is not defined by $\{x_1,\dots, x_m\}$. In addition, because each source fiber $\CG x$ is infinite and the partial action is free, for any $x_i$ there is a $g_i\in \Gamma$ such that $g_ix_i\in X\setminus \{x_1,\dots, x_m\}$. Now either define $T_{n, i}=F_ng_i$ if $g^{-1}_i\in F_n$ or define $T_{n, i}=\{e_\Gamma\}\cup (F_n\setminus \{h_i\})g_i$ for some arbitrary $h_i\in F_n$ when $g^{-1}_i\notin F_n$. Observe that in either case $|T_{n, i}\Delta F_ng_i|\leq 2$ and $|T_{n, i}|=|F_n|$.
	 Then for each $x_i$, one can choose a clopen neighborhood $N_i$ such that
	 \begin{enumerate}
	 	\item  each $(T_{n, i}, N_i)$ is a tower,
	 	\item $\{N_i: 1\leq i\leq m\}$ is a disjoint family and
	 	\item $\sup_{\mu\in M(\CG)}\mu(N_i)<1/nm$ by \cite[Lemma 6.13]{M-W}.
	 \end{enumerate} 
  Then choose a clopen partition $\{O_1, \dots, O_k\}$ of $X\setminus \bigsqcup_{i=1}^mN_i$ such that $(F_n, O_j)$ is a tower for any $1\leq j\leq k$. Finally define 
  \[S_n=\bigcup_{j=1}^k\{(\gamma x, \gamma. x): \gamma\in F_n, x\in O_j\}\cup \bigcup_{i=1}^m\{(\gamma x, \gamma. x): \gamma\in T_{n, i}, x\in N_i\}\]
  and   easy to verify that $\{S_n: n\in \N\}$ is a F{\o}lner sequence in which all $S_n$ is $(1, 1/n)$-good and has the $1$-controlled height. Therefore, $\CG$ is almost finite in measure.
 \end{eg}

\section{Uniform property $\Gamma$}
In this section we provide the main application of our almost finiteness in measure by establishing Theorem B that  reduced $C^*$-algebra $C^*(\CG)$ has the uniform property $\Gamma$ for a minimal principal second countable groupoid $\CG$, which is almost finite in measure.
This result is motivated by the same result in \cite{D-G} for transformation groupoids. However, as we explains in the introduction, their method cannot be simply generalized to deal with general almost finite in measure groupoids. Instead, we use the technique called \textit{extendability} and \textit{almost elementariness} introduced in \cite{M-W} to establish the result. Recall the following definition of the \textit{uniform property $\Gamma$} introduced in \cite{C-E-T-W}.
\begin{defn}\cite[Definiton 2.1 and Proposition 2.2]{C-E-T-W}\label{defn:uniform gamma}
Let $A$ be a separable $C^*$-algebra with a compact $T(A)\neq \emptyset$. Then $A$ is said to have uniform property $\Gamma$ if for any finite subset $\CF\subset A$, any $\epsilon>0$, and any $n\in \N$, there exist pairwise orthogonal positive contractions $e_1,\dots, e_n\in A$ such that for $i=1,\dots, n$ and $a\in \CF$, one has $\|[e_i, a]\|<\epsilon$ and $\sup_{\tau\in T(A)}|\tau(ae_i)-(1/n)\tau(a)|<\epsilon$.
\end{defn}
Let $\CC=\{C_{i,j}^{l}:i,j\in F_{l},l\in I\}$ be a castle and $K$
be a compact set in $\CG$ with $\CG^{(0)}\subset K$. We say that
$\CC$ is $K$-\emph{extendable} if there is another castle $\CD=\{D_{i,j}^{l}:i,j\in E_{l},l\in I\}$ with $\CC\subset \CD$
such that 
\[
K\cdot\bigsqcup_{i,j\in F_{l}}C_{i,j}^{l}\subset\bigsqcup_{i,j\in E_{l}}D_{i,j}^{l}
\]
where $E_{l}\subset F_{l}$ and $C_{i,j}^{l}=D_{i.j}^{l}$ if $i,j\in E_{l}$
for all $l=1,\dots m$. In this case, we also say that $\CC$ is $K$-extendable
to $\CD$. Now we have the following definition, which is a weak version of \cite[Definition 6.9]{M-W}. In addition, since all groupoid considered in this paper are ample, we may ask the castles below to be compact open. See \cite[Proposition 6.12]{M-W}.
\begin{defn}
	\label{defn: ae in measure} Let $\CG$ be a groupoid
	with a compact unit space. We say that $\CG$ is \textit{almost elementary in measure}
	if for any compact set $K$ satisfying $\CG^{(0)}\subset K\subset\CG$,
	any open cover $\CV$ and $\epsilon>0$, there
	are compact open castles $\CC=\{C_{i,j}^{l}:i,j\in F_{l},l\in I\}$ and $\CD=\{D_{i,j}^{l}:i,j\in E_{l},l\in I\}$
	satisfying 
	\begin{enumerate}[label=(\roman*)]
		\item $\CC$ is $K$-extendable to $\CD$; 
		\item every $\CD$-level is contained in an open set $V\in\CV$; 
		\item $\sup_{\mu\in M(\CG)}\mu(\CG^{(0)}\setminus\bigsqcup_{l\in I}\bigsqcup_{i\in F_{l}}C_{i,i}^{l})<\epsilon$. 
	\end{enumerate}
\end{defn}

We remark that we mainly focus on the groupoids $\CG$ such that $M(\CG)\neq \emptyset$. This is because all topological principal ample groupoids $\CG$ such that  $M(\CG)=\emptyset$ is almost elementary in measure automatically. This is because first, without loss of generality, one may assume $K=\bigcup_{i=1}^nO_i$ for some compact open bisections $O_1=\GU,\dots, O_n$.  Then since $\CG$ is topological principal, there are a $u\in \GU$ and $1=i_1<\dots< i_k\leq n$ satisfying that there is a small compact open neighborhood $W$ of $u$ in $\GU$ such that $\{r(O_{i_j}W): j=1, \dots, k\}$ is disjoint and each $r(O_{i_j}W)\subset V$ for some $V\in \CV$ as well as $KW=\bigsqcup_{j=1}^kO_{i_j}W$. Then we define $\CC=\{W\}$ and $\CD=\{O_{i_j}WO^{-1}_{i_l}: 1\leq j,l\leq k\}$, which satisfy Definition \ref{defn: ae in measure}. On the other hand, when $\CG$ is fiberwise amenable, in which case $M(\CG)\neq \emptyset$, we have the following result, which was essentially established in \cite{M-W}.  
\begin{prop}\label{thm: ae in measure equal af in measure}
 Let $\CG$ be a $\sigma$-compact groupoid. Then if $\CG$ is almost finite in measure then $\CG$ is almost lelemenatry in measure and fiberwise amenable. If $\CG$ is additionally assumed to be minimal and second countable, the converse also holds.
\end{prop}
\begin{proof}
	The first part of the following theorem can be established by using the same proof of \cite[Theorem 7.4]{M-W} by deleting the final three lines. The second part can be shown by the first half of the proof of \cite[Proposition 7.6]{M-W}, i.e.,  the elementary groupoid $\CH_{\CC'}$ for the castle $\CC'$ defined in the proof of \cite[Proposition 7.6]{M-W} satisfying  $|K\CH_{\CC'}u|<(1+\epsilon)|\CH_{\CC'}u|$ for any $u\in \HU_{\CC'}$ and $\sup_{\mu\in M(\CG)}\mu(\GU\setminus\HU_{\CC'})<\epsilon$ for a prescribed compact set $K\subset \CG$ and $\epsilon>0$ as desired.
	\end{proof}

Then we proceed as in Section 8 of \cite{M-W} in a virtually identical way. In fact, the only difference between almost elementariness in measure in Definition \ref{defn: ae in measure} and the original almost elementariness introduced in \cite{M-W} is how to evaluate the ``smallness'' of the reminder beyond the castle $\CC$. The almost elementariness ask that the reminder is dynamical small, which is stronger than the smallness uniformly in measure in  Definition \ref{defn: ae in measure} for almost elementary in measure. Therefore all results in the Section 8 of \cite{M-W} for almost elementariness can be reproduced to obtain their ``in measure'' versions by deleting the \emph{groupoid strict comparison} (see \cite[Definition 6.2]{M-W}) in their proofs. Thus we have the following nesting version of almost finiteness in measure written in the language of extenability by proposition \ref{thm: ae in measure equal af in measure} and the discussion above. Compare it to Theorem 8.12 in \cite{M-W}. 

\begin{prop}\label{prop:nesting}
	Let $\CG$ be a  minimal groupoid. Suppose $\CG$ is
	almost finite in measure. Then for any compact set $K=\bigcup_{i=0}^nM_i$ in which all $M_i$ are compact open bisections and $M_0=\GU$,  any $\epsilon>0$,
	any open cover $\CV$ and any integer $N\in\N$ there are compact open castles
	$\CA$, $\CB$, $\CC$ and $\CD$ such that 
	\begin{enumerate}
		\item for any $0\leq i\leq n$ and $\CC$-level $C$, either $C\subset s(M_{i})$
		or $C\cap s(M_{i})=\emptyset$ and 
		\item whenever a $\CC$-level $C\subset s(M_{i})$ for some $i\leq n$ then
		there is a $D\in\CD$ such that $s(D)=C$ and $M_{i}C=D$. 
		\item $\CA$ is $K$-extendable to $\CB$ and $\CC$ is $K$-extendable
		to $\CD$; 
		\item $\CB$ is nested in $\CD$ with multiplicity at least $N$ in the sense of \cite[Definition 8.8]{M-W}; 
		\item $\CA$ is nested in $\CC$ with multiplicity at least $N$ in the sense of \cite[Definition 8.8]{M-W}; 
		\item any $\CD$-level is contained in a member of $\CV$; 
		\item $\sup_{\mu\in M(\CG)}\mu(\GU\setminus\bigcup\CA^{(0)})<\epsilon$. 
	\end{enumerate}
	\qed
\end{prop}
Then we recall a fundamental construction introduced in \cite{M-W} that will be the main tool to establish Theorem B.
\begin{rmk}\label{rmk: ccp map}
	\label{8.8} Let $\CG$ be a groupoid. Let $n\in\N$ and $\epsilon>0$. In addition,
	let $N\in\N$ such that $N>2/\epsilon$ and $K$ be a compact set
	in $\CG$. Suppose $\CA,\CB,\CC$ and $\CD$ are compact open castles such
	that 
	\begin{enumerate}[label=(\roman*)]
		\item $\CA$ is $K$-extendable to $\CB$ and $\CC$ is $K$-extendable
		to $\CD$. 
		\item $\CB$ is nested in $\CD$ with multiplicity at least $nN$. 
		\item $\CA$ is nested in $\CC$ with multiplicity at least $nN$. 
		\item $\mu(\bigcup\CA^{(0)})>1-\epsilon/2$ for any $\mu\in M(\CG)$. 
	\end{enumerate}
	If we write $\CA$ and $\CB$ explicitly, say, by $\CA=\{A_{i,j}^{l}:i,j\in F_{l},l\in I\}$
	and $\CB=\{B_{i,j}^{l}:i,j\in E_{l},l\in I\}$ where $F_{l}\subset E_{l}$
	for each $l\in I$.
	Then consider the collection $\{1_{B}:B\in\CB\}$ of functions in $C_{c}(\CG)$ satisfying
	\begin{enumerate}[label=(\roman*)]
		\item $s(1_{B})=1_{s(B)}$ and $r(1_{B})=1_{r(B)}$ for each $B\in\CB$. 
		\item $1_{r(B)}*1_{B}=1_{B}$ and $1_{B}*1_{s(B)}=1_{B}$ for each $B\in\CB$. 
	\end{enumerate}
Now we write $\CC$ and $\CD$ explicitly by $\CC=\{C_{t,s}^{p}:t,s\in T_{p},p\in J\}$
	and $\CD=\{D_{t,s}^{p},t,s\in S_{p},p\in J\}$. Let $\HU\subset\CD^{(0)}$
	be a subset containing $\CC^{(0)}$. Now, since $\CUU\subset\HU$,
	one has that $\HU$ contains some $\CD$-levels from multisection
	$(\CD^{p})^{(0)}$ for any $p\in J$. Denote by $(\CH^{p})^{(0)}=(\CD^{p})^{(0)}\cap\HU$.
	Now, for each $p\in J$,  let $l\in I$ such that $\CB^{l}$ is
	nested in $\CD^{p}$ with multiplicity at least $nN$. Fix a level
	$D_{t,t}^{p}$ where $t\in S_{p}$ and define $P_{p,t,l}=\{B\in\BUl:B\subset D_{t,t}^{p}\}$.
	Note that $|P_{p,t,l}|\geq nN$. Then for $m=1,\dots,n$ choose a
	subset $P_{p,t,l,m}\subset P_{p,t,l}$ such that $|P_{p,t,l,m}|=\floorstar{|P_{p,t,l}|/n}.$
	In addition, choose a bijection $\Lambda_{p,t,l,m}:P_{p,t,l,1}\to P_{p,t,l,m}$. Then for any $D\in(\CH^{p})^{(0)}$ there
	is a bisection $D'\in\CD^{p}$ such that $s(D')=D_{t,t}^{p}$ and
	$r(D')=D$. Now define 
	\[
	P_{D,l}=\{r(D'B):B\in P_{p,t,l}\}
	\]
	and $P_{D,l,m}=\{r(D'B):B\in P_{p,t,l,m}\}$ for all $1\leq m\leq n$.
	In addition, define maps $\Theta_{D,l,m}:P_{D,l,1}\to P_{D,l,m}$
	by 
	\[
	\Theta_{D,l,m}(r(D'B))=r(D'\Lambda_{p,t,l,m}(B))
	\]
	for any $B\in P_{p,t,l,1}$ and define $\Theta_{D,l,k,m}=\Theta_{D,l,k}\circ\Theta_{D,l,m}^{-1}$.
	From this construction, for any $p\in J$ and $l\in I$ such that
	$\CB^{l}$ is nested in $\CD^{p}$, we actually have the following
	configuration. 
	\begin{enumerate}[label=(\roman*)]
		\item $P_{D,l}=\{B\in\BUl:B\subset D\}$ has the cardinality $|P_{D,l}|>nN$
		for any $D\in\HUp$. 
		\item There are collections $P_{D,l,m}\subset P_{D,l}$ such that $|P_{D,l,m}|=\floorstar{|P_{D,l}|/n}$
		for any $D\in\HUp$ and $1\leq m\leq n$. 
		\item There are bijective maps $\Theta_{D,l,k,m}:P_{D,l,m}\to P_{D,l,k}$
		for any $D\in\HUp$ and $1\leq m,k\leq n$. For any $1\leq k,m,p\leq n$,
		these functions also satisfy 
		
		\begin{enumerate}
			\item $\Theta_{D,l,m,m}$ is the identity map; 
			\item $\Theta_{D,l,k,m}^{-1}=\Theta_{D,l,m,k}$; 
			\item $\Theta_{D,l,k,m}\Theta_{D,l,m,p}=\Theta_{D,l,k,p}$. 
		\end{enumerate}
		
		. 
		
		\item For any $D\in\CD$ such that $s(D),r(D)\in\HUp$ one has 
		\[
		r(D\Theta_{s(D),l,k,m}(B))=\Theta_{r(D),l,k,m}(r(DB))
		\]
		for any $B\in P_{s(D),l,m}$ and $1\leq k,m\leq n$. 
	\end{enumerate}
	In this case, we call such a collection of all sets $P_{D,l,m}$ together
	with all maps $\Theta_{D,l,k,m}$ for any $p\in J$, $l\in I$ such
	that $\CB^{l}$ is nested in $\CD^{p}$, $D\in\HUp$, $1\leq k,m\leq n$,
	a $\HU$-$\BU$-\textit{nesting system}.
	
	Now for $D\in\HUp$, $l\in I$ such that $\CB^{l}$ is nested in $\CD^{p}$
	and $1\leq k,m\leq n$, we denote by 
	\[
	R_{D,l,k,m}=\{B\in\CB^{l}:s(B)\in P_{D,l,m}\ \text{and}\ r(B)=\Theta_{D,l,k,m}(s(B))\}.
	\]
	For each $p\in J$ write $I_{p}=\{l\in I:\CB^{l}\ \text{is nested in }\CD^{p}\ \text{with multiplicity at least }nN.\}$
	and for each $D\in\HUp$ define 
	\[
	Q_{k,m,D}=\bigsqcup_{l\in I_{p}}R_{D,l,k,m}.
	\]
	In addition, we fix an arbitrary function $\kappa:\HU\to[0,1]$ such that $\kappa(C)=1$ for any $C\in \CUU$. Denote
	by $e_{km}$ the matrix in $M_{n}(\C)$ whose $(k,m)$-entry is $1$
	while other entries are zero. Now we define a map $\psi:M_{n}(\C)\to C_{r}^{*}(\CG)$
	by 
	\[
	\psi(e_{km})=\sum_{D\in\HU}\sum_{B\in Q_{k,m,D}}\kappa(D)1_{B}
	\]
	and is linearly extended to define on the whole $M_{n}(\C)$, which is in fact a c.p.c.\ order
	zero map by \cite[Lemma 9.5]{M-W}.
	
	On the other hand, note that for each $p\in J$ the index set $I_{p}$
	consists exactly all $l\in I$ such that $\CA^{l}$ is nested in $\CC^{p}$
	with multiplicity at least $nN$. Then for any $\CC^{p}$-level $C_{t,t}^{p}$
	and $l\in I_{p}$, there are at most $n-1$ levels $A\in\AUl$ with
	$A\subset C^p_{t,t}$ so that $\psi(1_{n})$ is not supported on. Then choose
	one such level, denoted by $A_{p,l}$, in $(\CA^l)^{(0)}$. Now, for any $\mu\in M(\CG)$,
	first the fact that $\CA^{l}$ is nested in $\CC^{p}$ with multiplicity
	at least $nN$ implies that 
	\[
	\sum_{p\in J}\sum_{l\in I_{p}}nN|T_{p}|\mu(A_{p,l})\leq\mu(\bigcup\AU)\leq1.
	\]
	Then one has 
	\begin{align*}
	\mu(\bigcup\{A\in\CA^{(0)}:\psi(1_{n})\equiv0\ \text{on}\ A\})\leq\sum_{p\in J}\sum_{l\in I_{p}}(n-1)|T_{p}|\mu(A_{p,l})\leq1/N.
	\end{align*}
	and thus 
	\[
	\mu(\bigcup\{A\in\CA^{(0)}:\psi(1_{n})\equiv1\ \text{on}\ A\})\geq1-\epsilon/2-1/N\geq1-\epsilon,\]
	Then for $1\leq m\leq n$, define $g_m=\psi(e_{mm})$ and $g'_m=g_m|_{\bigcup \CUU}=g_m|_{\bigcup \AU}$ in $C(\GU)$. By our construction, one has $\sum_{m=1}^n\mu(g'_m)\geq 1-\epsilon$ for any $\mu\in M(\CG)$.
	In addition, given a collection $M$ consisting of some $\CD$-levels. Then one has $\mu(g_m|_{\bigcup M})=\mu(g_k|_{\bigcup M})$ and $\mu(g'_m|_{\bigcup M})=\mu(g'_k|_{\bigcup M})$ for any $\mu\in M(\CG)$ and $1\leq k, m\leq n$ because the number of $\CB$-levels and $\CA$-levels in any $D\in M$ such that each $g_m$ and $g'_m$ supported on are same. In addition, observe that $g_mg_k=0$ for each $1\leq m\neq k\leq n$ because their supports $\supp(g_m)$ for $1\leq m\leq n$ are disjoint. 
\end{rmk}
Now we prove the main result in this section.
\begin{thm}\label{thm: uniform property Gamma}
Let $\CG$ be a second countable minimal principal groupoid, which is almost finite in measure, then its reduced $C^*$-algebra $C^*_r(\CG)$ has the uniform property $\Gamma$.
\end{thm}
\begin{proof}
	Let $\CF$ be a finite subset of $C^*_r(\CG)$ and $\epsilon>0$. It suffices to find positive contractions $g_1, \dots g_n\in C(\GU)_+$ satisfy conditions in Definition \ref{defn:uniform gamma}. Without loss of generality, we may assume all $f\in \CF$ are supported on a compact open bisections $O_f$. In addition, note that for any tracial state $\tau\in T(A)$, since $\CG$ is principal, there is a $\mu_\tau\in M(\CG)$ such that $\tau(a)=\int E(a)\, d\mu_\tau$ for any $a\in A$, which implies that $\tau(g_if)-(1/n)\tau(f)=\tau(g_iE(f))-(1/n)\tau(E(f))$ for any $g_i$ and $f\in \CF$ mentioned above. Then observe there is a clopen partition $\CP=\{P_1, \dots, P_m\}$ of $\GU$ such that each $E(f_i)$ can be approximated with arbitrary small error by $\sum_{j=1}^{m}c^i_j1_{P_j}$ for some suitable scalar $c^i_j\in \C$. Therefore, to show $A$ has uniform property $\Gamma$, it suffices to show for any finite set $\CF$ in $C^*(\CG)$ in which all $f\in \CF$ is supported on a compact open bisection $O_f$, any $\epsilon>0$ and any clopen partition $\CP$ of $\GU$ there are positive contractions  $g_1,\dots, g_k\in C(\GU)_+$ such that $\|[g_i, f]\|<\epsilon$ and $|\mu(g_i1_P)-(1/n)\mu(1_P)|<\epsilon$ for any $g_i, f\in \CF$ and $P\in \CP$.
	
	Then define the clopen cover $\CO^1_f=\{s(O_f), \GU\setminus s(O_f)\}$ and $\CO^2_f=\{r(O_f), \GU\setminus r(O_f)\}$ for any $f\in \CF$ and the compact set \[K=\bigcup_{i=1}^{N+1}\bigcup_{f_1, \dots, f_i\in \CF}\{U_{f_1}\cdot U_{f_2}\cdots U_{f_i}: U_{f_j}=O_{f_j}\ \text{or }O^{-1}_{f_j}, j=1,\dots, i\}\cup \GU\]
	and choose a clopen partition $\CV$ of $\GU$, which is finer than $\CP$ as well as $\CO^1_f$ and $\CO^2_f$ for all $f\in \CF$ such that for any $V\in \CV$ and any $u, v\in V$ and $f\in \CF$ one has 
	\[|s(f)(u)-s(f)(v)|<\epsilon\ \text{and}\ |r(f)(u)-r(f)(v)|<\epsilon.\] 
	Then Proposition \ref{prop:nesting} implies that there are compact open castles $\CA$, $\CB$, $\CC$ and $\CD$ such that 
	\begin{enumerate}
		\item $\CA$ is $K$-extendable to $\CB$ and $\bar{\CC}$ is $K$-extendable
		to $\bar{\CD}$; 
		\item For any $i\leq N+1$, $f_{1},\dots,f_{i}\in F$ and $\CC$-level $C$, either $C\subset s(U_{f_{1}}\cdot U_{f_{2}}\dots\cdot U_{f_{i}})$ or $C\cap s(U_{f_{1}}\cdot U_{f_{2}}\dots\cdot U_{f_{i}})=\emptyset$,  where
		$U_{f_{k}}=O_{f_{k}}$ or $U_{f_{k}}=O_{f_{k}}^{-1}$ for any $1\leq k\leq i$.
		\item For any $i\leq N+1$ and $f_{1},\dots,f_{i}\in F$ if a $\CC$-level
		$C\subset s(U_{f_{1}}\cdot U_{f_{2}}\dots\cdot U_{f_{i}})$ where
		$U_{f_{k}}=O_{f_{k}}$ or $U_{f_{k}}=O_{f_{k}}^{-1}$ for any $1\leq k\leq i$
		then $U_{f_{1}}\cdot U_{f_{2}}\dots\cdot U_{f_{i}}C=D$ for some $D\in\CD$. 
		\item $\CB$ is nested in $\CD$ with multiplicity at least $nN$ 
		\item $\CA$ is nested in $\CC$ with multiplicity at least $nN$; 
		\item any $\CD$-level is contained in a member of $\CV$; 
		\item $\mu(\GU\setminus\bigcup\CA^{(0)})<\epsilon$ for any $\mu\in M(\CG)$. 
	\end{enumerate}
Then we define the following sets consisting of certain $\CD$-levels. First define  Now define $\DU_{0}=\CUU_{N}$
and inductively define 
\begin{align*}
\DU_{k}=\{ & D\in\DU:D=r(U_{f_{k}}\dots U_{f_{1}}C),U_{f_{i}}=O_{f_{i}}\ \text{or}\ O_{f_{i}}^{-1},\ \text{for}\ i=1,\dots,k,\\
& f_{1},\dots,f_{k}\in F\ \text{and}\ C\in\CUU\}\setminus\bigsqcup_{i=0}^{k-1}\DU_{i}
\end{align*}
for $k=1,\dots,N+1$ (some $\DU_{k}$ may be empty). Define $\HU=\bigsqcup_{k=0}^{N}\DU_{k}$,
which is a subset of $\DU$ and contains $\CUU=\DU_{0}$. So far we have finished the basic setting of our construction. Then in the same way of the proof of \cite[Theorem 9.7]{M-W}, using the construction describe in Remark \ref{rmk: ccp map}, one can choose a $\HU$-$\BU$-nesting system and define an order zero c.p.c map  $\psi:M_{n}(\C)\to C_{r}^{*}(\CG)$
by 
\[
\psi(e_{ij})=\sum_{D\in\HU}\sum_{B\in Q_{i,j,D}}\kappa(D)h_{B}
\]
and extending linearly in which 
$\kappa:\HU\to[0,1]$ is defined by $\kappa(D)=1-k/N$ if $D\in\DU_{k}$
for $k=0,\dots,N$. Then for $1\leq m\leq n$ define orthogonal positive contractions $g_m=\psi(e_{mm})\in C(\GU)_+$ as well as $g'_m=g_m|_{\bigcup \CUU}=g_m|_{\bigcup \AU}$ in $C(\GU)$.
In addition, Remark \ref{rmk: ccp map} implies that 
\begin{enumerate}
	\item For any $P\in \CP$, if write $M_P=\{D\in \DU: D\subset P\}$ then $\mu(g_m|_{\bigcup M_P})=\mu(g_k|_{\bigcup M_P})$ and $\mu(g'_m|_{\bigcup M_P})=\mu(g'_k|_{\bigcup M_P})$ for any $\mu\in M(\CG)$ and $1\leq k, m\leq n$.
	\item $\sum_{m=1}^n\mu(g'_m)\geq 1-\epsilon$ for any $\mu\in M(\CG)$.
\end{enumerate}
Therefore,  one first observe that $\mu(g_m1_P)=\mu(g_m|_{\bigcup M_P})$ for each $m\leq n$ and thus \[n\mu(g_m1_P)=\mu((\sum_{m=1}^ng_m)1_P)\leq \mu(1_P)\] for any $P\in \CP$, $1\leq m\leq n$ and $\mu\in M(\CG)$. On the other hand, one also has 
\begin{align*}
	\mu(1_P)-n\mu(g_m1_P)&\leq \mu(1_P)-n\mu(g'_m1_P)\\
	&=\mu(P\setminus\bigcup\{A\in \AU: A\subset C\ \text{for some }C\in M_P\})\\
	&\leq \mu(\GU\setminus \bigcup \AU)<\epsilon
\end{align*}
for any $\mu\in M(\CG)$. These implies that
$|\mu(g_m1_P)-(1/n)\mu(1_P)|<\epsilon/n<\epsilon$ for any $\mu\in M(\CG)$.
Finally, the same argument in \cite[Theorem 9.7]{M-W} for each $g_m=\psi(e_{mm})$ shows that $\|[g_m, f]\|<\epsilon$ for any $1\leq m\leq n$ and $f\in \CF$. This finishes our proof.
\end{proof}
As a direct application,  using \cite[Theorem 5.6]{C-E-T-W}, one has the following results.

\begin{cor}\label{cor: tw conjecture}
Let $\CG$ be a minimal (topologically) amenable second countable ample groupoid which is also almost finite in measure. Then $C^*_r(\CG)$ satisfies the Toms-Winter conjecture. This thus applies to minimal partial dynamical systems in \ref{eg: partial dynamical system}.
\end{cor}

\section{Soficity of topological full groups}
In this section, we use our fiberwise amenability to study topological full groups of groupoids. Our aim is to establish Theorem C and Corollary D. We recall the definition of topological full groups first. Denote by $\aut(\GU)$ the set of all homeomorphism from $\GU$ to itself. Let $B$ be a compact open bisection in $\CG$ and write $\pi_B=r\circ (s|_B)^{-1}$ for the natural homeomorphism from $s(B)$ to $r(B)$.

\begin{defn}
	Let $\CG$ be agroupoid. We denote by $[[\CG]]$ the subgroup \[\{\pi_B\in \aut(\GU): \text{B is a compact open bisection such that }s(B)=r(B)=\GU\}\] 
	of $\aut(\GU)$ the topological full group of $\CG$.
\end{defn}

The definition of sofic groups is originally due to Weiss and Gromov. All amenable discrete groups and residually finite groups are sofic and it is still unknown whether there is a non-sofic group. 
Denote by $\map(A)$ for a finite set $A$ all maps from $A$ to itself and define the \emph{Hamming distance} on $\map(A)$ by $d_H(f, g)=|\{x\in A: f(x)\neq g(x)\}|/|A|$. We record the following definition of sofic groups appeared in \cite{El}
\begin{defn}[\cite{El}]
	A group $\Gamma$ is said to be sofic if for any finite $F\subset \Gamma$ and $\epsilon>0$ there exists a finite set $A$ and a mapping $\Theta: \Gamma\to \map(A)$ such that 
	\begin{enumerate}
		\item If $f, g, fg\in F$ then $d_H(\Theta(fg), \Theta(f)\Theta(g))\leq \epsilon$.
		\item If $ \id_A\neq f\in F$ then $d_H(\Theta(f), \id_A)>1-\epsilon$.
		\item $\Theta(\id_A)=\id_A$.
	\end{enumerate}
\end{defn}

Motivated by the fact that countable discrete amenable groups are sofic, it is reasonable to conjecture that $[[\CG]]$ is sofic when $\CG$ is fiberwise amenable and second countable.  We confirmed this in Theorem \ref{thm: sofic} when $\CG$ is minimal and admits a F{\o}lner sequence. To establish this, one needs to apply \emph{compressed sofic representation} proved in \cite[Lemma 2.1]{El} because $[[\CG]]$ has many fixed units in $\GU$.  The following was proved by Elek in \cite[Lemma 2.1]{El}.

\begin{prop}[\cite{El}]\label{prop: compressed sofic representation}
	Let $\{e_\Gamma=\gamma_0, \gamma_1,\dots, \}$ be an enumeration of a countable discrete group $\Gamma$. Suppose for any $i\geq 1$ there is a constant $\epsilon_i>0$ and for any $n\geq 1$ there is an map $\Theta_n: \Gamma\to \map(A_n)$ for some finite set $A_n$ with $\Theta_n(e_\Gamma)=\id_{A_n}$ and satisfying the condition that for all $r>0$ and $\epsilon>0$ there exists $K_{r, \epsilon}>0$ such that  if $n>K_{r, \epsilon}$ one always has
	\begin{enumerate}
	\item $d_H(\Theta_n(\gamma_i\gamma_j), \Theta_n(\gamma_i)\Theta_n(\gamma_j)<\epsilon$ if $1\leq i, j\leq r$.
	\item $d_H(\Theta_n(\gamma_i), \id_{A_n})>\epsilon_i$ if $1\leq i\leq r$.
	\end{enumerate}
Then $\Gamma$ is sofic.
\end{prop}

 We begin with the following lemma, which yields finite sets $A_n$ satisfying Proposition \ref{prop: compressed sofic representation}.

\begin{lem}\label{lem: folner enough}
	Let $\CG$ be a groupoid. Let $F=\{\pi_{V_1}, \dots, \pi_{V_n}\}$ be a finite subset of $[[\CG]]$ and $K=(\bigcup_{1\leq i, j\leq n} V_iV_j)\cup (\bigcup_{1\leq i\leq n}V_i)\cup \GU$. Let $S\subset \CG_u$ be a $(K, \epsilon)$-F{\o}lner set such that $r|_S$ is injective. Write $P=r(S)$. Then $|\bigcap_{g\in F\cup F^2}g^{-1}(P)\cap P|\geq (1-\epsilon)|P|$.
\end{lem}
\begin{proof}
	Let $M=\{\gamma\in S: K\gamma\subset S\}=S\setminus \partial^-_KS$, which satisfies $|M|\geq (1-\epsilon)|S|$ because $S$ is $(K, \epsilon)$-F{\o}lner. Then observe that $r(M)\subset \{u\in P: (F^2\cup F)u\subset P\}=\bigcap_{g\in F\cup F^2}g^{-1}(P)\cap P$. Thus one has 
	\[|\bigcap_{g\in F\cup F^2}g^{-1}(P)\cap P|\geq |r(M)|=|M|\geq (1-\epsilon)|S|=(1-\epsilon)|P|.\] 
\end{proof}

Let $g\in [[\CG]]$. We denote by $\supp(g)$ the clopen set $\{u\in \GU: g(u)\neq u\}$.  Now we are ready to show the following theorem.

\begin{thm}\label{thm: sofic}
Let $\CG$ be a second countable minimal groupoid that admits a F{\o}lner sequence. Then $[[\CG]]$ is sofic.
\end{thm}
\begin{proof}
	First, let $\CS=\{S_n: n\in \N\}$ be a F{\o}lner sequence. Then for any non-empty compact open set $A$ in $\GU$, one has $\barbelow{D}^-_\CS(A)\geq \inf_{\mu\in M(\CG)}\mu(A)$ by Lemma \ref{lem: density less measure}. We still write \[D^-_{S_n}(A)=\inf_{u\in \GU}(1/|S_nu|)\sum_{\gamma\in S_nu}1_A(r(\gamma))\]
	and thus $\barbelow{D}^-_\CS(A)=\liminf_{n\to \infty}D^-_{S_n}(A)$ by definition.
	Now, enumerate $[[\CG]]$ by $\{\gamma_0=\id_{\GU}, \gamma_1, \gamma_2, \dots\}$. Since $\CG$ is minimal, writing $A_i=\supp(\gamma_i)$ for any $i\in \N^+$, there is an $\epsilon_i>0$ such that $\inf_{\mu\in M(\CG)}\mu(A_i)\geq 3\epsilon_i$.
	
	For $A_1$, choose a subsequence $\{S^1_n: n\in \N\}=\{S_{k_n}: n\in \N\}$ of $\{S_n: n\in\N\}$ such that $\barbelow{D}^-_\CS(A_1)=\lim_{n\to \infty}D^-_{S^1_n}(A_1)$. Now suppose one has defined from $1$ until $k$ such that $\{S^{i+1}_n: n\in \N\}$ is a subsequence of $\{S^{i}_n: n\in \N\}$ for all $1\leq i\leq k-1$ and satisfying
	\[\lim_{n\to \infty}D^-_{S^i_n}(A_i)\geq \barbelow{D}^-_\CS(A_i)\]
	 for all $1\leq 1\leq k$. Then for $k+1$, observe that $\liminf_{n\to \infty}D^-_{S^k_n}(A_{k+1})\geq \barbelow{D}^-_\CS(A_{k+1})$ because $\{S^k_{n}: n\in\N\}$ is still a subsequence of the original F{\o}lner sequence $\CS=\{S_n: n\in \N\}$. Then choose a subsequence $\{S^{k+1}_n: n\in \N\}$ of $\{S^k_n: n\in \N\}$ such that 
	 \[\lim_{n\to \infty}D^-_{S^{k+1}_n}(A_{k+1})=\liminf_{n\to \infty}D^-_{S^k_n}(A_{k+1})\geq \barbelow{D}^-_\CS(A_{k+1}).\]
	This finishes our construction by induction. Now we pick up the diagonal sequence $\{T_n=S^n_n: n\in \N\}$, which satisfies that
	\begin{enumerate}
		\item $\{T_n: n\in \N\}$ is still a F{\o}lner sequence and $\{T_n: n\in \N\}\subset \CS$.
		\item $\lim_{n\to \infty}D^-_{T_n}(A_i)\geq \barbelow{D}^-_\CS(A_i)\geq \inf_{\mu\in M(\CG)}\mu(A_i)\geq 3\epsilon_i$ for all  $i\geq 1$.
	\end{enumerate}
For each $n\in \N$, we fix a $u_n\in \GU$ and define $P_n=r(T_nu_n)$. Note that $r|_{T_nu_n}$ is injective. Then define $\theta_n: [[\CG]]\to \sym(P_n)$ as follows. For $\varphi\in [[\CG]]$,  define a map $\sigma_{\varphi, n}$  from $\varphi^{-1}(P_n)\cap P_n$ to $P_n\cap \varphi(P_n)$ by $\sigma_{\varphi, n}(u)=\varphi(u)$, which is well-defined and bijective. In addition, we fix another arbitrary bijective map $\rho_{\varphi, n}: P_n\setminus \varphi^{-1}(P_n)\to P_n\setminus \varphi(P_n)$.  Finally, we announce that our $\theta_n(\varphi)$ is defined to be the combination map of  $\sigma_{\varphi, n}$ and $\rho_{\varphi, n}$, which belongs to $\sym(P_n)$. Note that $\theta_n(\id_{\GU})=\id_{P_n}$ by our definition. We then verify that our all $\theta_n$ satisfy the assumptions of Proposition \ref{prop: compressed sofic representation}.

Let $r\in \N, \epsilon>0$. First, choose a $0<\delta<\min\{\epsilon, \epsilon_i: 1\leq i\leq r\}$. Then there is a big enough number $K_{r, \epsilon}>0$ such that whenever $n>K_{r, \epsilon}$, the set $T_nu_n$ is F{\o}lner enough with respect to $F=\{\gamma_0, \dots, \gamma_r\}$ and $\delta$ such that
$|\bigcap_{g\in F\cup F^2}g^{-1}(P_n)\cap P_n|\geq (1-\delta)|P_n|$
by Lemma \ref{lem: folner enough} and
\[D^-_{T_n}(A_i)=\inf_{u\in \GU} (|A_i\cap r(T_nu)|/|T_nu|)\geq\barbelow{D}^-_\CS(A_i)-\delta>3\epsilon_i-\delta>2\epsilon_i\]
for all $i\leq r$.

Now, let $1\leq i\leq r$, which means $\gamma_i\in F\setminus \{\id_{\GU}\}$. Then the inequality above implies that $|A_i\cap P_n|/|P_n|\geq D^-_{T_n}(A_i)>2\epsilon_i$ because $|T_nu_n|=|P_n|$.  Observe that whenever $u\in \gamma_i^{-1}(P_n)\cap P_n\cap A_i$, one has $\theta_n(\gamma_i)(u)=\gamma_i(u)\neq u$ because $u\in A_i=\supp(\gamma_i)$. On the other hand, note that \[|\gamma_i^{-1}(P_n)\cap P_n\cap A_i|> 2\epsilon_i|P_n|-\delta|P_n|\geq\epsilon_i|P_n|,\]
which implies that $|\{u\in P_n: \theta_n(\gamma_i)(u)\neq u\}|>\epsilon_i|P_n|$.

Finally, let $0\leq i, j\leq r$, we write $\varphi=\gamma_i\in F$ and $\psi=\gamma_j\in F$ for simplicity. Note that for any $u\in \varphi^{-1}(P_n)\cap \psi^{-1}(P_n)\cap \psi^{-1}\varphi^{-1}(P_n)\cap P_n$, one always has $\theta_n(\varphi\psi)(u)=\varphi(\psi(u))$ by definition as well as $\theta_n(\psi)(u)=\psi(u)\in P_n\cap \varphi^{-1}(P_n)$, which implies  $\theta_n(\varphi)(\theta_n(\psi)(u))=\varphi(\psi(u))$. This implies that 
\[\varphi^{-1}(P_n)\cap \psi^{-1}(P_n)\cap \psi^{-1}\varphi^{-1}(P_n)\cap P_n\subset \{u\in P_n: \theta_n(\varphi)\theta_n(\psi)(u)=\theta(\varphi\psi)(u)\}\]
and thus $|\{u\in P_n: \theta_n(\varphi)\theta_n(\psi)(u)=\theta_n(\varphi\psi)(u)\}|>(1-\delta)|P_n|\geq (1-\epsilon)|P_n|$.

Thus Proposition \ref{prop: compressed sofic representation} shows that $[[\CG]]$ is sofic.
\end{proof}

We have the following natural corollaries.

\begin{cor}\label{cor 1}
Let $\CG$ be a second countable minimal fiberwise amenable groupoid. Suppose there is a $u\in \GU$ such that $\CG^u_u=\{u\}$. Then $[[\CG]]$ is sofic. In particular,  if $\CG$ is topological principal then the alternating group $A(\CG)$ introduced in \cite{Ne} is simple and sofic. 
\end{cor}
\begin{proof}
The proof is a direct application of Proposition \ref{prop: minimal groupoid normal folner}, \ref{prop: exist folner sequence}, Theorem \ref{thm: sofic} and Theorem 1.1 in \cite{Ne}.
\end{proof}

The following strengthens Matui's result on commutator group $D[[\CG]]$ of a minimal almost finite groupoid $\CG$

\begin{cor}\label{cor 2}
	if $\CG$ is second countable minimal almost finite then $D[[\CG]]$ is simple and sofic.  
\end{cor}
\begin{proof}
	This is a direct application of Proposition \ref{prop: almost finite folner sequence}, \ref{prop: exist folner sequence}, Theorem \ref{thm: sofic} and  Theorem 4.7 in \cite{Matui2}.
\end{proof}

We end this paper with the following remark including several examples.

\begin{rmk}
	\begin{enumerate}
		\item 	if we apply Corollary 7.6 to minimal Cantor dynamical systems of amenable groups $\alpha: \Gamma\curvearrowright X$, one actually obtain that the topological full group $[[X\rtimes \Gamma]]$ is sofic when there is a free point $x\in X$. However, combining results in \cite{E-L}, \cite{El} and \cite{Pau}, $[[X\rtimes \Gamma]]$  is always sofic in the most general sense, i.e., neither necessarily $\alpha$ is minimal nor  there are free points. The method is to study full groups of sofic equivalence relations introduced in \cite{E-L} and \cite{Pau} in a measurable setting. 
		
		\item Corollary 7.8 naturally applies to many minimal almost finite groupoids, such as certain minimal almost finite geometric groupoids introduced in \cite{E} and groupoids of repetitive aperiodic tillings of $\R^d$ with finite local complexity, e.g., Penrose tillings, considered in \cite{I-W-Z}.
	\end{enumerate}

\end{rmk}

\section{Acknowledgement}
The author should like to thank David Kerr, Hanfeng Li and Yongle Jiang for answering his questions on Banach densities and sofic groups. In addition, he would like to thank Jianchao Wu for helpful discussions and comments.

\end{document}